\documentclass[11pt]{amsart}

\title[Divides and hyperbolic volumes]
{Divides and hyperbolic volumes}

\author{Ryoga Furutani}
\address{
Department of Mathematics~ \slash ~International Institute for Sustainability with Knotted Chiral Meta Matter (WPI-SKCM$^2$), 
Hiroshima University, 1-3-1 Kagamiyama, Higashi-Hiroshima, 739-8526, Japan}
\email{ryoga.furutani0409@gmail.com}

\author{Yuya Koda}
\address{
Department of Mathematics, Hiyoshi Campus, Keio University, 4-1-1, Hiyoshi, Kohoku, Yokohama, 223-8521, Japan~ \slash ~ 
International Institute for Sustainability with Knotted Chiral Meta Matter (WPI-SKCM$^2$), 
Hiroshima University, 1-3-1 Kagamiyama, Higashi-Hiroshima, 739-8526, Japan}
\email{koda@keio.jp}

\thanks{
Y. K. is supported by JSPS KAKENHI Grant Numbers JP20K03588, 
JP21H00978 and JP23H05437. 
}

\usepackage{amsthm}
\usepackage{mathrsfs}
\usepackage{latexsym}
\usepackage[dvipdfmx]{graphicx}
\usepackage[dvips]{psfrag}
\usepackage[dvips]{color}
\usepackage{xypic}
\usepackage[abs]{overpic}
\usepackage{caption}
\usepackage[all]{xy}
\usepackage[normalem]{ulem}

\theoremstyle{plain}
\newtheorem*{theorem*}{Theorem}
\newtheorem*{lemma*} {Lemma}
\newtheorem*{corollary*} {Corollary}
\newtheorem*{proposition*}{Proposition}
\newtheorem*{conjecture*}{Conjecture}
\newtheorem{theorem}{Theorem}[section]
\newtheorem{lemma}[theorem]{Lemma}
\newtheorem{corollary}[theorem]{Corollary}
\newtheorem{proposition}[theorem]{Proposition}

\theoremstyle{remark}

\newtheorem*{definition}{Definition}
\newtheorem*{claim*}{Claim}

\newtheorem{example}{Example}

\theoremstyle{definition}

\newtheoremstyle{citing}
  {}
  {}
  {\itshape}
  {}
  {\bfseries}
  {.}
  {.5em}
  {\thmnote{#3}}

\theoremstyle{citing}

\newcommand{\NN}{\mathbb{N}}

\newcommand{\RR}{\mathbb{R}}
\newcommand{\CC}{\mathbb{C}}

\newcommand{\Nbd}{\operatorname{Nbd}}

\newcommand{\Int}{\operatorname{Int}}

\newcommand{\gl}{\operatorname{gl}}
\newcommand{\vol}{\operatorname{vol}}
\newcommand{\slope}{\operatorname{sl}}

\newcommand\erase{\bgroup\markoverwith{\textcolor{red}{\rule[.5ex]{2pt}{0.4pt}}}\ULon}



\begin{document}

\maketitle

\begin{abstract}
A divide is the image of a proper and generic immersion of a compact $1$-manifold into the $2$-disk. 
Due to A'Campo's theory, each divide is associated with a link in the 3-sphere. 
In this paper, we reveal a hidden hyperbolic structure in the theory of links of divides.  
More precisely, we show that the complement of the link of a divide can be obtained by Dehn filling 
a hyperbolic $3$-manifold that admits a decomposition into several ideal regular tetrahedra, octahedra and 
cuboctahedra, where the number of each of those three polyhedra is determined by types of the double points of the divide. 
This immediately gives an upper bound of the hyperbolic volume of the links of divides, which is shown to be asymptotically sharp. 
An idea from the theory of Turaev's shadows plays an important role here. 
\end{abstract}

\vspace{1em}

\begin{small}
\hspace{2em}  \textbf{2020 Mathematics Subject Classification}: 
32S55; 57K10, 57K32, 57Q60, 57R05


\hspace{2em} 
\end{small}

\section*{Introduction}

Let $D$ be a $2$-disk, and $W$ be the subspace of the tangent bundle $TD = D \times \RR^2$ of $D$ defined by 
$W = \{ (x,u) \in TD \mid x \in D,~ u \in T_{x}D,~|x|^{2} + |u|^{2} \leq 1\}$, which is a $4$-ball.
The image $P$ of a proper and generic immersion of a compact $1$-manifold into $D$ is called a \textit{divide}. 
Each divide $P$ defines a link $L_{P} := \{ (x,u) \in W \mid x \in P,~ u \in T_{x}P,~|x|^{2} + |u|^{2} = 1\}$, called the \textit{link of $P$},  
in the $3$-sphere $\partial W = S^3$. 
Divides and their links are introduced by A'Campo \cite{Camp98} as part of his study on complex plane curve singularities. 
The links of divides are actually a generalization of the links of singular points of complex plane curves, and in particular, 
when a divide is connected, its link is a fibered link.
In \cite {Camp98_2}, A'Campo defined a class of links of divides, called \textit{slalom knots}, and gave 
a necessary and sufficient condition for those knots to be hyperbolic, where 
a link $L \subset S^3$ is said to be \textit{hyperbolic} if $S^3 - L$ admits a complete hyperbolic structure of finite volume. 
Consequently, it was shown that they include many hyperbolic ones.
In this paper, we explain another hidden hyperbolic structure behind links of divides. 

Let $P \subset D$ be a connected divide. 
We call a connected component of $D - P$ a \textit{region} of $P$. 
A region of $P$ is said to be \textit{internal} if it does not intersect $\partial D$. 
Otherwise, $P$ is said to be \textit{external}. 
We define a $2$ dimensional polyhedron $X_P$ in the 4-ball $W$ to be the union of $2$-disk $D$ and 
$\{(x, u) \in W \mid x \in P,~u \in T_{x }P\}$ minus all external regions of $P$.
This polyhedron $X_P$ is properly embedded in $W$, and $\partial X_P$ is exactly the link $L_{P}$ of $P$. 
Furthermore, the $4$-ball $W$ collapses onto $X_P$ in a natural way.
Let $\pi : \partial W \to X_P$ be the projection obtained by restricting this collapse $W \searrow X_P$ to the boundary $\partial W$ of a $4$-ball $W$. 
Let $D_1, D_2, \ldots, D_n$ be the closures of the connected components of $X_P - \Nbd (P; X_P)$ contained in internal regions of $P$, which are all $2$-disks. 
Set $M_P := \pi^{-1} (\Nbd (P; X_P))$ and $V_i := \pi^{-1} (D_i)$ ($i= 1 , 2 , \ldots, n$). 
Each $V_i$ here is a solid torus.
Then, the exterior $E(L_{P})$ of the link $L_{P}$ can be decomposed as 
$ E(L_P) = M_P \cup ( \bigcup_{i=1}^n V_i )$.
In other words, 
$E(L_{P})$ can be obtained by performing Dehn fillings on the $3$-manifold $M_P$. 
The purpose of this paper is to clarify the geometric (hyperbolic) structure of $M_P$.

Suppose that the divide $P$ is prime. Here, a divide $P$ is said to be \textit{prime} 
if the link $L_P$ of $P$ is prime. 
In this case, the vertices of  the polyhedron $X_P$ can be classified into the six types shown in 
Figure \ref{figure:local_models_of_XP_prime}. 
We will show that we can decompose the preimage of a neighborhood of each vertex 
into  several truncated ideal regular hyperbolic polyhedra, 
so that we get an ideal hyperbolic regular polyhedral decomposition of the interior of $M_P$.  
In fact, the following is our main theorem: 
\begin{theorem}
\label{thm:introduction main theorem}
Let $P \subset D$ be a connected prime divide with at least one double point. 
If $P$ has a double point of Type $6$-$3$, then $L_P$ is the Hopf link. 
Otherwise, let $n_1$, $n_2$, $n_3$, $n_4$ and $n_5$ be the number of its double points of 
Types $1$, $2$, $3$, $4$-$2$ and $5$-$3$, respectively.  
Then $\Int M_P$ is a hyperbolic $3$-manifold of volume 
\[ 10 n_3 v_{\mathrm{tet}} + (4n_1 + 2n_4 + n_5) v_{\mathrm{oct}} + n_2 v_{\mathrm{cuboct}}. \]
Here, $v_{\mathrm{tet}} = 1.014 \ldots$, $v_{\mathrm{oct}} = 3.663 \ldots$ and $v_{\mathrm{cuboct}} = 12.046 \ldots$ are 
the volumes of ideal hyperbolic regular tetrahedron, octahedron and cuboctahedron, respectively. 
\end{theorem}
In order to get the above decomposition, we use the idea from the theory of shadows introduced by Turaev. 
A shadow is a $2$-dimensional simple polyhedron $X$ properly and locally-flatly embedded in a $4$-manifold $W$ so that 
$W$ collapses onto $X$. 
Each region of a shadow is assigned a half integer called a \textit{gleam}, which is a kind of relative Euler number. 
From a shadow with gleams, we can reconstruct the $4$-manifold $W$ in a canonical way. 
Costantino and Thurtston \cite{CT08} revealed that the intersection of $\partial W$ and 
the preimages of a regular neighborhood of a vertex of a shadow by the collapsing map 
can be decomposed into two truncated ideal regular hyperbolic octahedra, and the octrahedra for the vertices are 
glued together according to the combinatorial structure of $X$ to form a hyperbolic 3-manifold of volume 
$2 n v_{\mathrm{oct}}$, where $n$ is the number of vertices of $X$. 

As an immediate consequence of Theorem \ref{thm:introduction main theorem}, we have the following upper bound 
of the volumes of the links of divides, which, in turn, is proved to be asymptotically sharp in Theorem \ref{thm: the upper boune is sharp}. 
\begin{corollary}
\label{cor:introduction main}
Let $P \subset D$ be a connected prime divide with at least one double point. 
Let $n_1$, $n_2$, $n_3$, $n_4$ and $n_5$ be the number of its vertices of 
Types $1$, $2$, $3$, $4$-$2$ and $5$-$3$, respectively.  
Then the hyperbolic volume of $L_P \subset S^3$ is less than 
\[ 10  n_3 v_{\mathrm{tet}} + (4n_1 + 2n_4 + n_5) v_{\mathrm{oct}} + n_2 v_{\mathrm{cuboct}}, \]
where the hyperbolic volume of a non-hyperbolic links is defined to be zero. 
\end{corollary}

\section{Preliminaries}

\subsection{Divides}
Let $D \subset \RR^2$ be a $2$-disk. 
A divide $P$ is the image of a proper, generic immersion from a compact $1$-manifold into $D$. 
That is, $P$ satisfies the following: 
\begin{itemize}
\item
$\partial D \cap P = \partial P$, and $P$ intersects $\partial D$ transversely at these points; and 
\item
The self intersections of $P$ are transverse, and there are no triple points. 
\end{itemize}
For a divide $P$, a point where $P$ intersects $\partial D$ is called an \textit{endpoint} of $P$. 
A connected component of $D - P$ is called a \textit{region} of $P$. 
A region that has no intersection with $\partial D$ is called an \textit{internal region}. 
Otherwise, a region is called an \textit{external region}.

We identify the tangent bundle $TD$ of the $2$-disk $D$ with $D \times \RR^2 \subset \RR^4$. 
Note that under this identification, $W := \{ (x, u) \in TD \mid x \in D, u \in T_x D, |x|^2+|u|^2 \leq 1 \} \subset \RR^4$ is the 4-ball. 
The \textit{link} $L_P$ of a divide $P$ is then defined by 
\[ L_P := \{ (x, u) \in W  \mid x \in P, u \in T_x P , |x|^2+|u|^2 = 1 \}  \subset \partial W \cong S^3 . \]
The existence of a natural involution 
$\iota : \partial W \to \partial W$, $(x, u) \mapsto (x, -u)$ shows that the link $L_P$ of a divide is strongly invertible. 
When the divide $P$ is connected, its link $L_P$ is a fibered link (see A'Campo \cite{Camp98}).

\begin{example}
\label{ex:divide}
Let $P$ be the divide shown on the left top in Figure \ref{figure:divide_example}. 
Then, according to the method described by Hirasawa \cite{Hira02}, we obtain the diagram of the link $L_P$ of $P$ as follows. 
First, we perturb $P$ so that it becomes a line segment with a slope of $1$ or $-1$ piecewise as shown on the left in Figure \ref{figure:divide_example}.
\begin{figure}[htbp]
\centering\includegraphics[width=13cm]{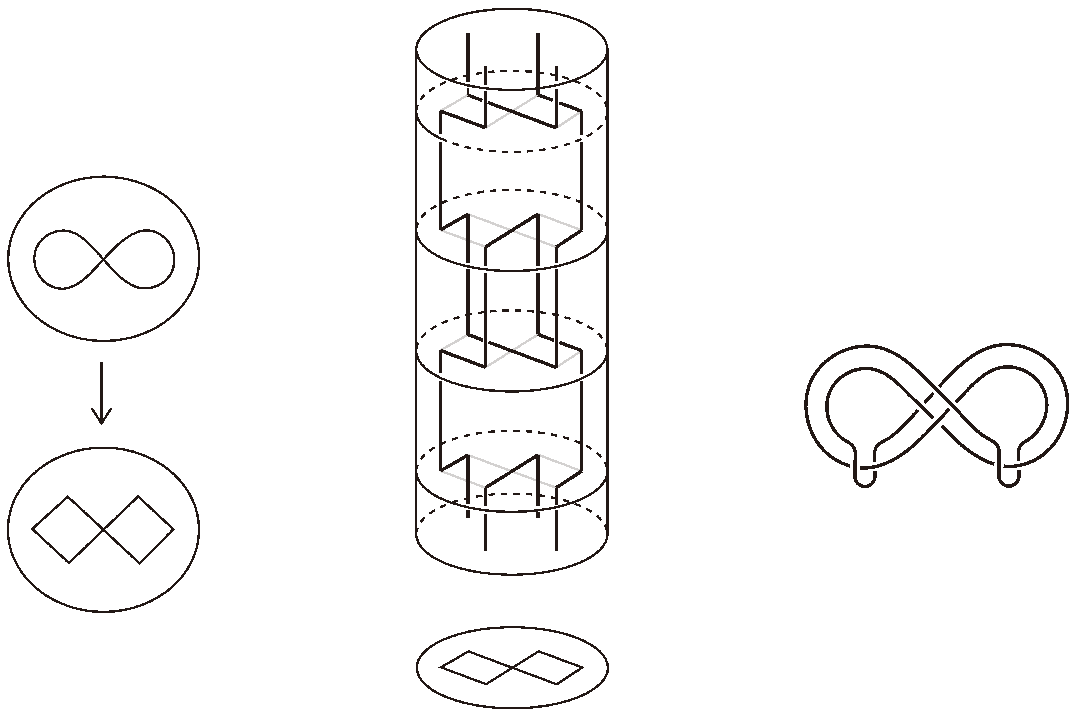}
\begin{picture}(400,0)(0,0)
\put(230,242){$\pi$}
\put(230,220){$\frac{3 \pi}{4}$}
\put(230,179){$\frac{\pi}{4}$}
\put(230,136){$\frac{- \pi}{4}$}
\put(230,97){$\frac{- 3 \pi}{4}$}
\put(230,70){$-\pi$}
\put(0,165){$P$}
\put(53,123){perturbing}
\put(332,75){$L_P$}
\end{picture}
\caption{(Left) A perturbation of $P$. (Middle) The preimage $h^{-1} (L_P)$. (Right) A diagram of $L_P$.}
\label{figure:divide_example}
\end{figure}
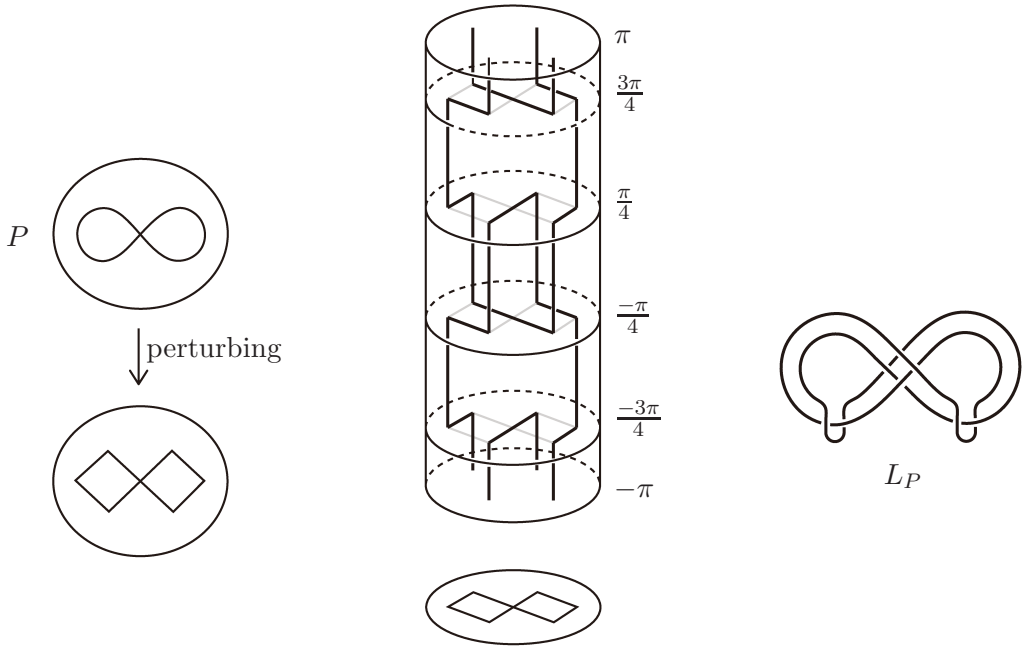
Define an equivalence relation $\sim$ on the cylinder $D \times [-\pi, \pi]$ by 
imposing $( x, \pi ) \sim (x, - \pi)$ (for $x \in D$) and $(x, \theta) \sim (x,0)$ 
 (for $x \in \partial D$, $\theta \in [ - \pi , \pi ]$). 
Note that the quotient space $D \times [- \pi , \pi] / \sim$ is homeomorphic to the 3-sphere. 
Indeed, we have a homeomorphism $h : (D \times [ - \pi , \pi] ) / \sim) \to \partial W$, $[(x , \theta)] \mapsto (x, \sqrt{1-|x|^{2}}e^{\sqrt{-1}\theta})$, 
where $\RR^4$ is identified with $\RR^2 \times \CC$, and thus, $\partial W$ is thought of as a subspace of $\RR^2 \times \CC$. 
The preimage $h^{-1} (L_P)$ of the link $L_P$ can be drawn on $D \times [ - \pi , \pi ] / \sim$ as shown on the middle in Figure \ref{figure:divide_example}. 
A diagram of $L_P$ is then obtained by projecting the preimage $h^{-1}(L_P)$ onto the disk $[D \times \{0 \}] \subset D \times [ - \pi , \pi ] / \sim$ 
after perturbing slightly $h^{-1}(L_P)$ so that the projection is generic. 
See the right in Figure \ref{figure:divide_example}.

\centering


\end{example}

\subsection{Shadows}
\label{subsec:shadow}

Let $X$ be a polyhedron, that is, the underlying space of a simplicial complex $K$. 
An (open) simplex $\sigma$ of $K$ is said to be \textit{free} if there exists a unique (open) simplex having $\sigma$ as its proper face. 
A point of a free simplex is called a \textit{boundary point} of the polyhedron $X$, and 
$\partial X$ denotes the set of boundary points of $X$. 
A simplex $\tau$ is said to be \textit{principal} if $\tau$ is not a proper face of any simplex of $K$. 
Let $\sigma$ be a free $k$-simplex of $K$ that is a proper face of  a principal $(k+1)$-simplex $\tau$. 
Then an operation to get $K - \{ \sigma, \tau \}$ form $K$ is called an \textit{elementary simplicial collapse}. 
Let $Y$ be a subpolyhedron of a polyhedron $X$. 
We say that $X$ \textit{collapses onto} $Y$, and denote $X \searrow Y$, if there exists a triangulation $(K, L)$ of the pair $(X,Y)$ such that 
 $L$ is obtained from $K$ by a finite sequence of elementary simplicial collapses.

A polyhedron $X$ is said to be \textit{simple} if each point $x$ in $X$ 
has a regular neighborhood homeomorphic to one of the four local models shown in Figure \ref{figure:simple_polyhedron}.
\begin{figure}[htbp]
\centering\includegraphics[width=10cm]{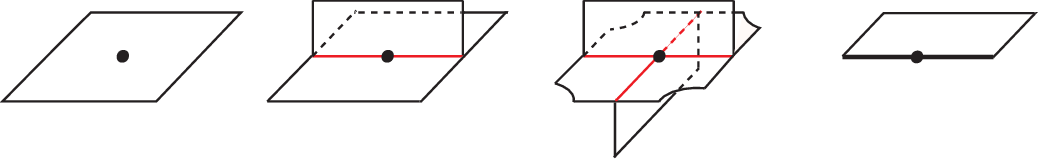}
\begin{picture}(400,0)(0,0)
\put(85,0){(i)}
\put(155,0){(ii)}
\put(228,0){(iii)}
\put(301,0){(iv)}
\end{picture}
\caption{Local models of a simple polyhedron.}
\label{figure:simple_polyhedron}
\end{figure}
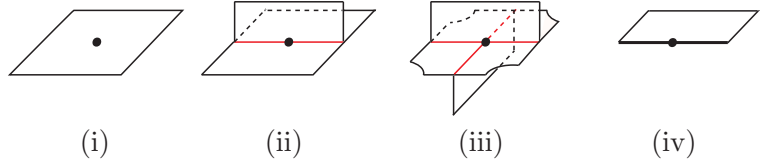
A point of $X$ having (ii) or (iii) in Figure \ref{figure:simple_polyhedron} as its regular neighborhood is called a \textit{singular point} of $X$, 
and a point of $X$ having (iii) as its regular neighborhood is called a \textit{vertex} of $X$. 
A point having (iv) as its regular neighborhood is called a \textit{boundary point} of $X$.
Throughout the paper, 
$S(X)$ denotes the set of singular points of $X$. 
A connected component of $X - S(X)$ is called a \textit{region} of $X$.
A region that has no intersection with $\partial X$ is called an \textit{internal} region.

Let $X$ be a simple polyhedron. 
A map from the set of connected components of $\partial X$ to the binary set $\{e, f \}$ is called a \textit{coloring} of $\partial X$.
Given a coloring of $\partial X$, we denote by $\partial_e X$ and $\partial_f X$ the preimages of $e$ and $f$, respectively, of the coloring. 
A simple polyhedron $X$ with a fixed coloring of $\partial X$ is called a \textit{boundary decorated} simple polyhedron.

\begin{definition}
\label{def:shadow}
Let $M$ be a compact orientable $3$-manifold whose boundary consists of (possibly empty) tori.
Then, a \textit{shadow} of $M$ is defined to be a boundary decorated simple polyhedron $X$ 
properly embedded in a compact, orientable (smooth) $4$-manifold $W$ with boundary so that
\begin{itemize}
\item
X is \textit{locally flat} in $W$.  
That is, for any point $x \in X$, there exists a local coordinate neighborhood 
$(U, \varphi)$ of $x \in W$ such that $\varphi (U \cap X) \subset \{ 0 \} \times \RR^3_+ \subset \RR^4_+$, 
and $\varphi (U \cap X)$ is homeomorphic to one of the models shown in Figure \ref{figure:simple_polyhedron} 
within $\{ 0 \} \times \RR^3_+$.
Here, $\RR^{n}_{+} := \{(x_{1},x_{2}, \ldots ,x_{n}) \in \RR^{n} \mid x_{n} \geq 0\}~(n =3,4)$. 
\item
When $W$ is given a natural PL structure, $W$ collapses into $X$.
\item
$M = \partial W - \Int \Nbd(\partial_{e}X; \partial W)$. 
\end{itemize}
\end{definition}

Let $W$ be a $4$-manifold and let $X \subset W$ be a shadow of a $3$-manifold $M = \partial W$.
Each internal region $R$ of $X$ is assigned a half-integer $\gl(R)$ called \textit{gleam} as in the following way. 
First, we fix a Riemannian metric on $W$.  
Let $\iota : R \hookrightarrow W$ be the inclusion.  
Let $\bar{R}$ be the metric completion of $R$ with the path metric inherited from a Riemannian metric on $R$. 
For simplicity, we assume that the natural extension $\bar{\iota} : \bar{R} \to M$ is injective, thus, 
we identify $\bar{\iota} ( \bar{R} ) $ with $\bar {R}$. 
Then, for each $x \in \partial \bar{R}$, the direction orthogonal to the region $\bar{R}$ is determined on the regular neighborhood 
$\Nbd (x ; X) \subset \RR^3_+$ of $x$.
This defines a $[-1, 1]$ bundle on $\partial \bar{R}$ in $W$. 
Let $\bar{R}'$ be a compact surface in $W$ obtained by perturbing $\bar{R}$ so that $\partial \bar{R}'$ is contained in the 
$[-1, 1]$ bundle over $\partial \bar{R}$, and $\bar{R}'$ is in a general position with respect to $\bar{R}$. 
Then the intersection of $\bar{R}$ and $\bar{R}'$ consists of a finite number of isolated points, 
and the gleam of $R$ is defined to be the following sum: 
\[ \gl(R) := \frac{1}{2} |\partial \bar{R} \cap \partial \bar{R}'| + |\Int \bar{R} \cap \Int \bar{R}'| . \]

For a shadow $X \subset W$, we denote the projection obtained by restricting the collapse $W \searrow X$ to the boundary $\partial W$ by $\pi : \partial W \to X$.
A boundary-decorated simple polyhedron equipped with gleams of the internal regions is called a 
\textit{shadowed polyhedron}. 
In \cite{Tur94} Turaev introduced a method for reconstructing $W$ from a shadowed polyhedron $X$.
This is called \textit{Turaev's reconstruction}.
An outline is as follows. 
Let $R_1, R_2, \ldots, R_m$ be the connected components of $X - \Int (\Nbd (S(X) ; X))$. 
Then we have the decomposition $X = \Nbd (S(X) ; X) \cup ( \bigcup_{i=1}^m R_i)$.
Furthermore, we can decompose $\Nbd(S(X); X)$ into several pieces each homeomorphic to (ii) or (iii) in Figure \ref{figure:simple_polyhedron}. 
By gluing the $3$-thickenings of these small pieces according to the combination structure of $X$, 
we obtain a $3$-thickening $M_{S(X)}$ of $\Nbd(S(X) ; X)$.
The $4$-thickening $W_{S(X)}$ is then obtained by taking the orientable $[-1, 1]$ bundle over $M_{S(X)}$.
For each $R_i$, take the orientable $[-1, 1]$ bundle $W_{R_i}$ on $R_i \times [-1, 1]$.
The 4-manifold $W$ is finally constructed by gluing these $W_{R_i}$'s 
to $W_{S(X)}$ using a gluing map specified by their gleams. 
Divides, reviewed in the previous subsection, and shadows, reviewed here, are strongly related. 
In fact, Ishikawa and Naoe explained in \cite{IN20} a method for, given a divide $P$, constructing a 
shadowed polyhedron of the exterior $E(L_P)$ of the link $L_P$ of a divide $P$. 

\vspace{1em}

Let $X \subset W$ be a shadow of a $3$-manifold $M$. 
Let $A_1$ be a piece of $X$ homeomorphic to the model  (iii) in Figure \ref{figure:simple_polyhedron}. 
Let $A_2$ be a piece of $X$ having a bigon with the gleam $+1$ as shown in Figure \ref{figure:piece_A2}.
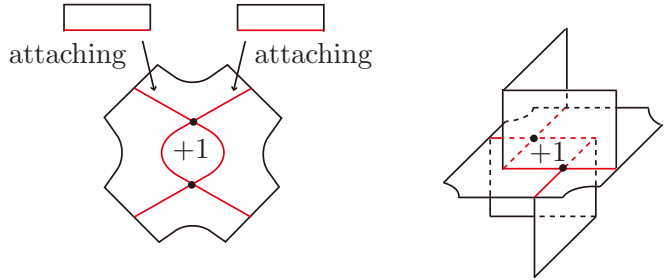
\begin{figure}[htbp]
\centering\includegraphics[width=8cm]{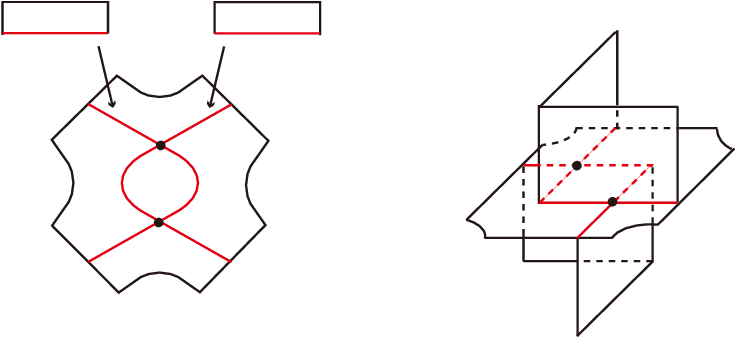}
\begin{picture}(400,0)(0,0)
\put(128,59){$+1$}
\put(263,58){$+1$}
\put(66,95){attaching}
\put(159,95){attaching}
\end{picture}
\caption{A piece $A_2$ of a shadow.}
\label{figure:piece_A2}
\end{figure}
Costantino and Thurston showed the following theorem by specifying the hyperbolic structure on the interior of 
the preimages $\pi^{-1} (A_1)$ and $\pi^{-1} (A_2)$. 

\begin{theorem}[Costantino-Thurston \cite{CT08}]
\label{prop:2v8+10v4}
Let $X$ be a simple polyhedron embedded in a compact $4$-manifold $W$ as a shadow of $M = \partial W$.
Let $X' \subset X$ be a connected simple polyhedron that can be decomposed into pieces each homeomorphic to $A_1$ or $A_2$.
Let $n_1$ and $n_2$ be the numbers of pieces homeomorphic to $A_1$ and $A_2$, respectively, in this decomposition.
Then $\Int  ( \pi^{-1} (X') )$ is a hyperbolic $3$-manifold whose hyperbolic volume is $2 n_1 v_{\mathrm{oct}} + 10 n_2 v_{\mathrm{tet}}$. 
Here, $v_{\mathrm{oct}}$ and $v_{\mathrm{tet}}$ are the volumes of the ideal hyperbolic regular  octahedron and the ideal hyperbolic regular tetrahedron, respectively.
\end{theorem}

Since the proof of Theorem \ref{prop:2v8+10v4} is closely related to the argument of our main theorem, we briefly review the outline here. 
By the assumption, $X'$ is decomposed into pieces each homeomorphic to $A_1$ or $A_2$.  
Let $C_1 \subset X'$ be a piece homeomorphic to $A_1$. 
Take the $3$-thickening $M_{C_{1}}$ of the local model $C_1$ as shown in Figure \ref{figure:face_shadow}. 
\begin{figure}[htbp]
\centering\includegraphics[width=8cm]{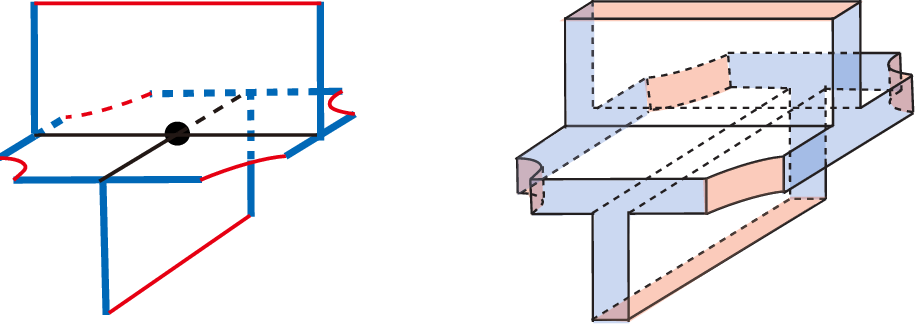}
\caption{A piece $C_{1}$ and its $3$-thickening $M_{C_{1}}$.}
\label{figure:face_shadow}
\end{figure}
Here, the blue part of the boundary $\partial C_{1}$ of $C_{1}$ in the figure 
indicates the part that is to be attached to other pieces 
of the decomposition of $X'$, while the red part indicates the rest. 
The blue and red faces of the boundary of $M_{C_{1}}$ indicate the parts corresponding to 
the blue and red parts of the boundary $\partial C_{1}$, respectively. 
We call the closure of a connected component of $\partial M_{C_{1}}$ minus the union of those parts a \emph{mirror face}. 
Due to Turaev's reconstruction, the preimage $\pi^{-1}(C_{1})$ of $C_{1}$ is obtained by gluing two copies of $M_{C_{1}}$ 
along the corresponding mirror faces. See Figure \ref{figure:two_octahedra_for_A1}.
Here, we can equip each of the two copies of $M_{C_{1}}$ with a hyperbolic structure 
by regarding the $3$-thickening $M_{C_{1}}$ as a truncated ideal hyperbolic regular octahedron.
In this way, the preimage $\pi^{-1}(C_1)$ can be thought of as a genus-3 handlebody obtained by gluing two
truncated ideal regular hyperbolic octahedra. 
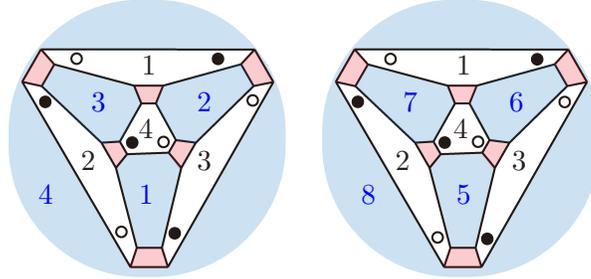
\begin{figure}[htbp]
\centering\includegraphics[width=8cm]{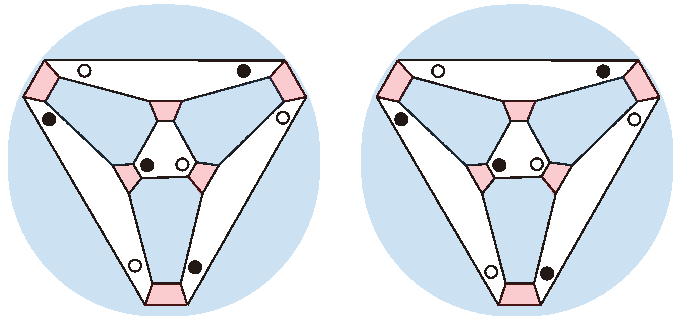}
\begin{picture}(400,0)(0,0)
\put(139,91){$1$}
\put(116,55){$2$}
\put(160,55){$3$}
\put(138,67){$4$}

\put(258,91){$1$}
\put(235,55){$2$}
\put(279,55){$3$}
\put(257,67){$4$}

\put(138,42){\color{blue}$1$}
\put(160,77){\color{blue}$2$}
\put(120,77){\color{blue}$3$}
\put(100,42){\color{blue}$4$}

\put(258,42){\color{blue}$5$}
\put(278,77){\color{blue}$6$}
\put(238,77){\color{blue}7$$}
\put(222,42){\color{blue}$8$}
\end{picture}
\caption{A polyhedral decomposition of $\pi^{-1} (C_1)$.}
\label{figure:two_octahedra_for_A1}
\end{figure}
Similarly, let $C_2 \subset X'$ be a piece homeomorphic to $A_2$. 
Then, the preimage $\pi^{-1} (C_2)$ can also be regarded as a genus-3 handlebody obtained by gluing ten truncated ideal regular hyperbolic tetrahedra 
as shown in Figure \ref{figure:ten_tetrahedra_for_A2} (see also Furutani-Koda \cite{FK21}).  
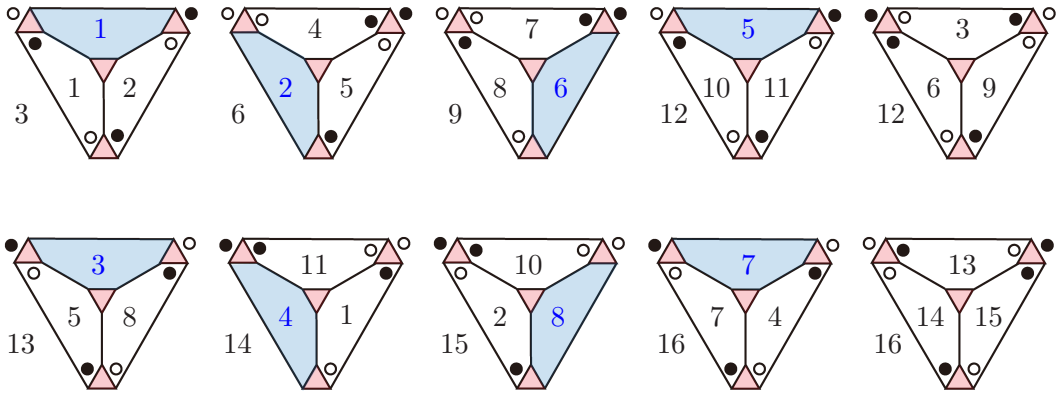
\begin{figure}[htbp]
\centering\includegraphics[width=14cm]{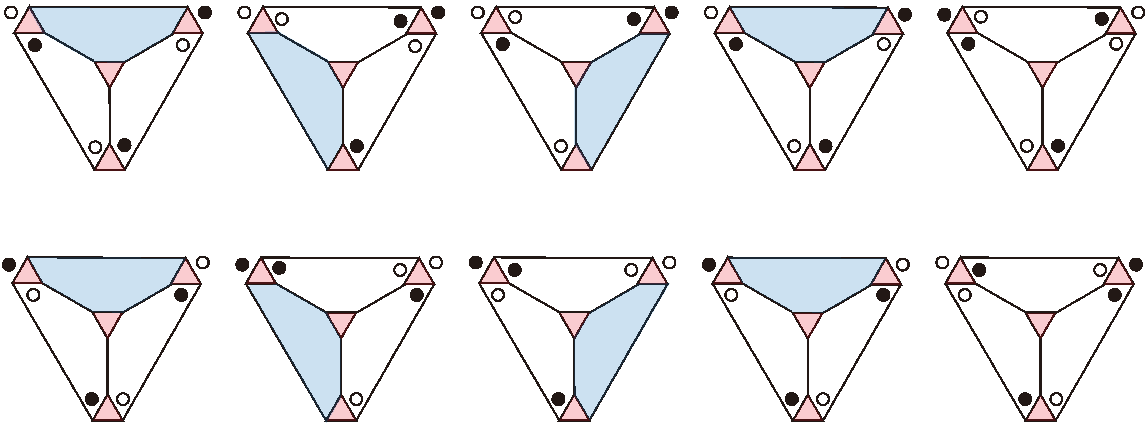}
\begin{picture}(400,0)(0,0)
\put(25, 125){$1$}
\put(128,39){$1$}
\put(46,125){$2$}
\put(186,38){$2$}
\put(5,115){$3$}
\put(361,147){$3$}
\put(116,147){$4$}
\put(290,38){$4$}
\put(128,125){$5$}
\put(25,38){$5$}
\put(87,115){$6$}
\put(350,125){$6$}
\put(198,147){$7$}
\put(268,38){$7$}
\put(186,125){$8$}
\put(46,38){$8$}
\put(169,115){$9$}
\put(371,125){$9$}
\put(265,125){$10$}
\put(194,58){$10$}
\put(288,125){$11$}
\put(113,58){$11$}
\put(249,115){$12$}
\put(331,115){$12$}
\put(358,58){$13$}
\put(2,28){$13$}
\put(346,38){$14$}
\put(84,28){$14$}
\put(368,38){$15$}
\put(166,28){$15$}
\put(248,28){$16$}
\put(330,28){$16$}

\put(35,147){\color{blue}$1$}
\put(105,125){\color{blue}$2$}
\put(209,125){\color{blue}$6$}
\put(280,147){\color{blue}$5$}
\put(34,58){\color{blue}$3$}
\put(105,38){\color{blue}$4$}
\put(208,38){\color{blue}$8$}
\put(280,58){\color{blue}$7$}
\end{picture}
\caption{A polyhedral decomposition of $\pi^{-1} (C_2)$.}
\label{figure:ten_tetrahedra_for_A2}
\end{figure}
In those figures, each white face is glued to another white face numbered the same so that 
the edges assigned $\circ$ and $\bullet$ match.  

Now, each of the genus-$3$ handlebodies $\pi^{-1} (C_1)$ and $\pi^{-1} (C_2)$  here 
has four three-holed spheres on its boundary, which are as shown in Figure \ref{figure:holed_spheres_for_A1}. 
\begin{figure}[htbp]
\centering\includegraphics[width=11cm]{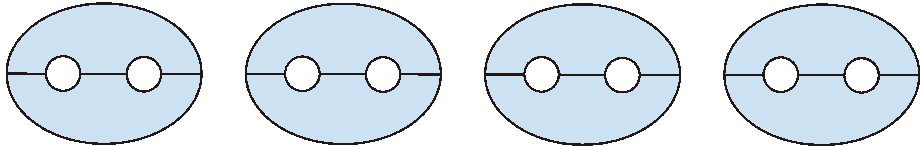}
\begin{picture}(400,0)(0,0)

\put(74,48){\color{blue}$1$}
\put(156,48){\color{blue}$2$}
\put(238,48){\color{blue}$3$}
\put(320,48){\color{blue}$4$}

\put(74,21){\color{blue}$5$}
\put(156,21){\color{blue}$6$}
\put(238,21){\color{blue}$7$}
\put(320,21){\color{blue}$8$}
\end{picture}
\caption{Four three-holes spheres on the boundary of $\pi^{-1} (C_i)$.}
\label{figure:holed_spheres_for_A1}
\end{figure}
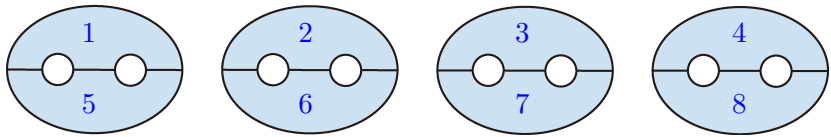
By gluing these three-holed spheres according to the combinatorial structure of $X'$, 
we obtain a decomposition of $\pi^{-1} (X')$ into truncated polyhedra, which induces 
a ideal hyperbolic polyhedral decomposition of $\Int  ( \pi^{-1} (X') )$ consisting of 
$2 n_1$ ideal regular hyperbolic octahedra and $10n_2$ ideal regular hyperbolic tetrahedra.

\begin{example}[Furutani-Koda \cite{FK21}]
\label{ex:link_10v4} 
Let $L$ be the link in $S^{3}$ shown on the left in Figure \ref{figure:shadow_of_L}.
\begin{figure}[htbp]
\centering\includegraphics[width=12cm]{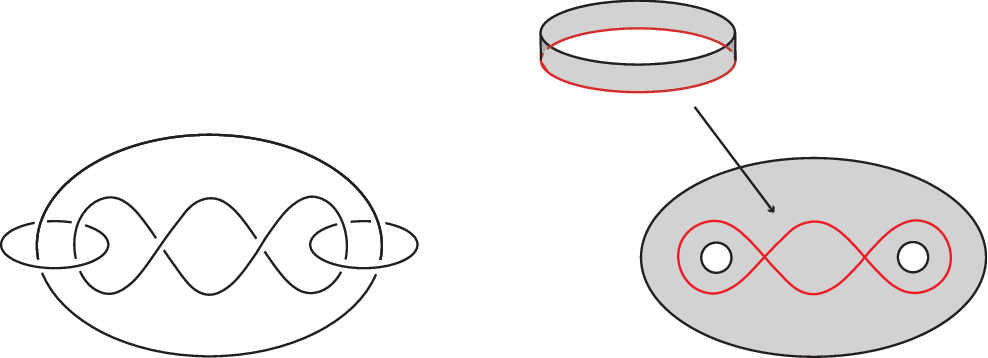}
\begin{picture}(400,0)(0,0)
\put(95,0){$L$}
\put(285,100){attaching}
\put(303,46){$+1$}
\end{picture}
\caption{The link $L$ and a shadow of $E(L)$.}
\label{figure:shadow_of_L}
\end{figure}
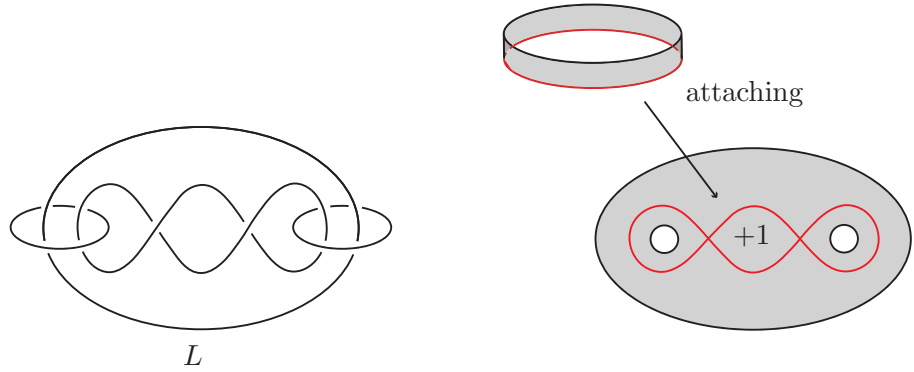
The exterior of $L$ has the shadowed polyhedron shown on the right in Figure \ref{figure:shadow_of_L} as its shadow.
This polyhedron is obtained by identifying the \textit{tripods}, i.e. the cones on three points, 
in the boundary of $A_2$ each other as shown on the right in Figure \ref{figure:cutting_along_pants_L}.
\begin{figure}[htbp]
\centering\includegraphics[width=12cm]{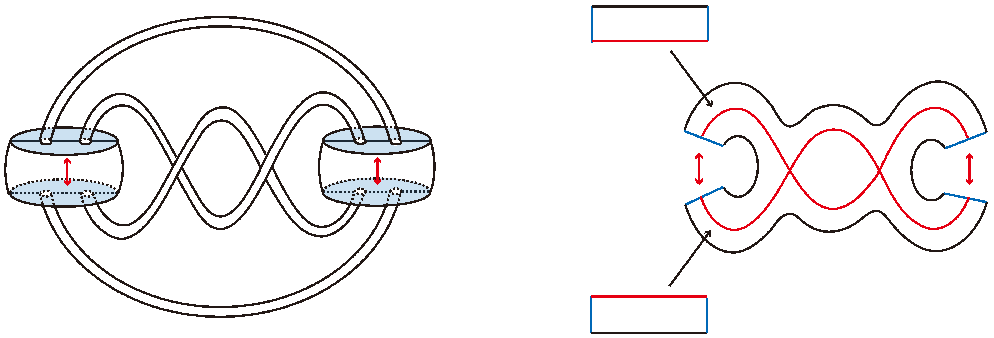}
\begin{picture}(400,0)(0,0)
\put(217,98){attaching}
\put(214,35){attaching}
\put(310,68){$+1$}
\end{picture}
\caption{From $A_2$ to a shadow of $E(L)$.}
\label{figure:cutting_along_pants_L}
\end{figure}
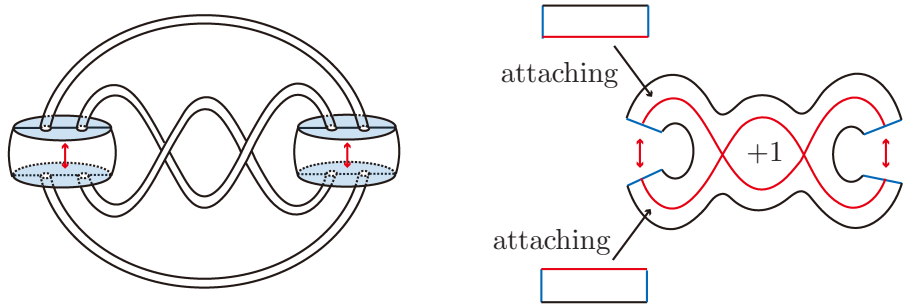
From Turaev's reconstruction, the exterior of $L$ is obtained by identifying 
the totally-geodesic $3$-holed spheres on the boundary of the genus-$3$ handlebody $\pi^{-1} (A_2)$, 
see the left in Figure \ref{figure:cutting_along_pants_L}. 
Thus, the exterior $E(L)$ of $L$ is decomposed into ten truncated regular tetrahedra as shown in Figure \ref{figure:ten_tetrahedra_for_L}. 
\begin{figure}[htbp]
\centering\includegraphics[width=14cm]{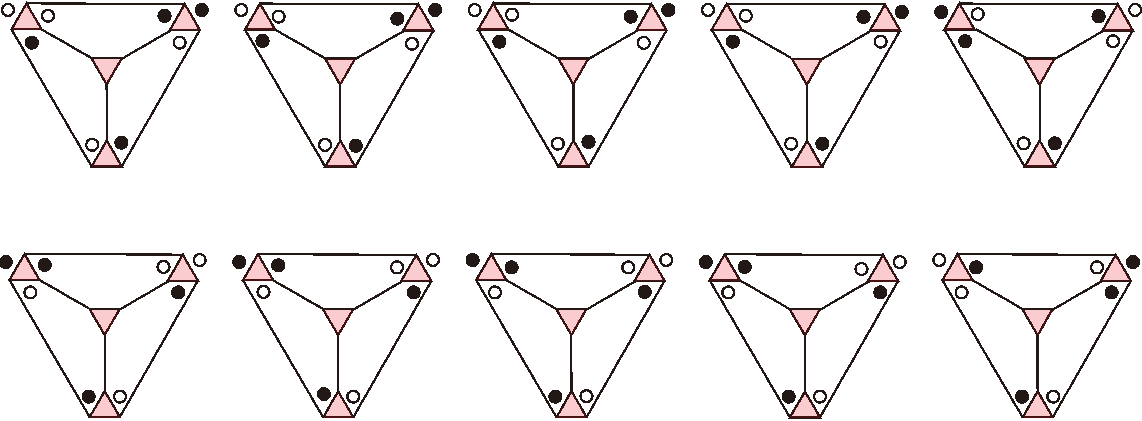}
\begin{picture}(400,0)(0,0)
\put(35,147){$1$}
\put(361,147){$4$}
\put(117,147){$5$}
\put(198,147){$9$}
\put(276,147){$13$}

\put(24, 125){$2$}
\put(47,125){$3$}
\put(103,125){$6$}
\put(128,125){$7$}
\put(184,125){$10$}
\put(207,125){$11$}
\put(264,125){$14$}
\put(287,125){$15$}
\put(350,125){$8$}
\put(369,125){$12$}

\put(5,115){$4$}
\put(87,115){$8$}
\put(166,115){$12$}
\put(249,115){$16$}
\put(331,115){$16$}

\put(32,58){$11$}
\put(113,58){$15$}
\put(194,58){$14$}
\put(276,58){$6$}
\put(358,58){$17$}

\put(24,38){$7$}
\put(43,38){$10$}
\put(99,38){$13$}
\put(128,38){$2$}
\put(186,38){$3$}
\put(208,38){$1$}
\put(267,38){$9$}
\put(291,38){$5$}
\put(345,38){$18$}
\put(368,38){$19$}

\put(2,28){$17$}
\put(84,28){$18$}
\put(166,28){$19$}
\put(248,28){$20$}
\put(330,28){$20$}

\end{picture}
\caption{A decomposition of $E(L)$ into ten truncated regular tetrahedra.}
\label{figure:ten_tetrahedra_for_L}
\end{figure}
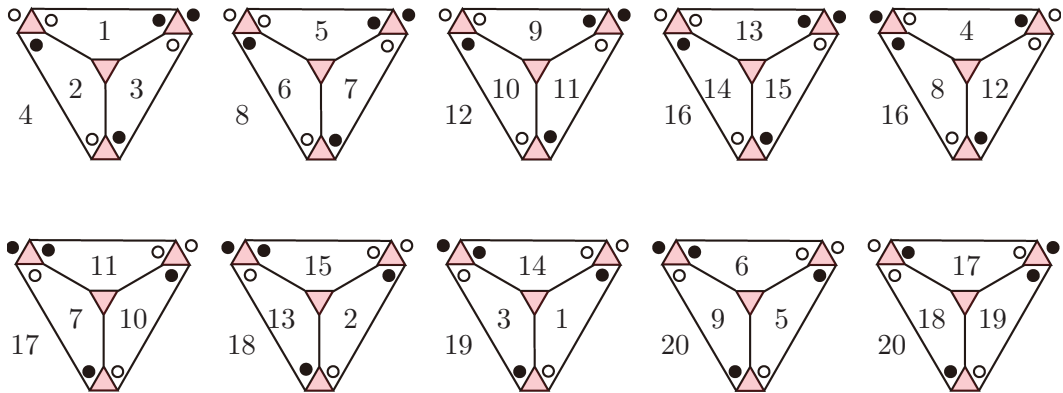
Since exactly six faces meet at each edge in this decomposition, this induces a decomposition of 
$\Int (\pi^{-1} (A_2))$ into ten ideal hyperbolic regular tetrahedra by assigning $\pi/3$ to each dihedral angle of the 
tetrahedra. 
Thus, we see that $L$ is a hyperbolic link whose hyperbolic volume is $10v_{\mathrm{tet}}$.
\end{example}

\section{The complements of links associated to divides}

Let $D \subset \RR^{2}$ be a $2$-disk. 
Set $W := \{(x, u) \in TD \mid x \in D, ~u \in T_{x}D, ~|x| ^{2} + |u|^{2} \leq 1\}$ as in the previous section. 
Let $P$ be a connected divide.

\begin{definition}\label{defi:xp}
Define the polyhedron $\widehat{X_{P}} \subset W$ to be the union of the disk $D \subset W$ and $\{(x,u) \in W \mid x \in P,~ u \in T_{x}P\}$. 
Besides, the polyhedron $X_{P}$ in $W$ is defined as $\widehat{X_{P}}$ minus the union of all external regions of the divide $P$. 
A point of $\widehat{X_{P}}$ or $X_{P}$ corresponding to a double point of $P$ is called a \textit{vertex}.
A connected component of $X_{P} - P$ is called a \textit{region} of $X_{P}$. 
\end{definition}

Remark that $ X_{P}$ is properly embedded in $W$, and $\partial X_P$ is nothing but the link $L_P$ of a divide.  
Furthermore, $W$ collapses onto $X_{P}$. 
The projection obtained by restricting this collapse $W \searrow X_P$ to the boundary $\partial W$ is denoted by $\pi : \partial W \to X_{P}$. 
We define the links $L_{X_P}$ and $L_{\widehat{X_P}}$ associated to a divide $P$ as follows.  
The polyhedra $\widehat{X_P}$, $X_P$, and $\Nbd(P; X_P)$ are as in Figure \ref{figure:hatXP_XP_Nbd}. 
\begin{figure}[htbp]
\centering\includegraphics[width=13cm]{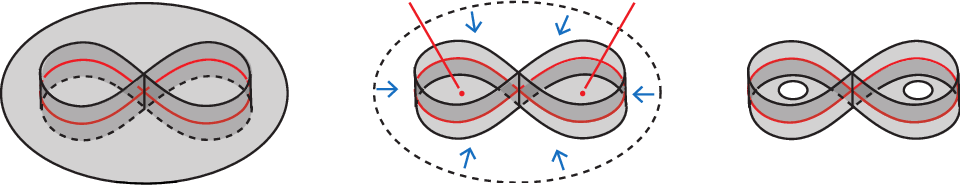}
\begin{picture}(400,0)(0,0)
\put(63,0){$\widehat{X_P}$}
\put(205,0){$X_P$}
\put(313,0){$\Nbd (P; X_P)$}

\put(164,88){\color{red} $x_1$}
\put(255,88){\color{red} $x_2$}
\end{picture}
\caption{The polyhedra $\widehat{X_P}$, $X_P$ and $\Nbd (P; X_P) $.}
\label{figure:hatXP_XP_Nbd}
\end{figure}
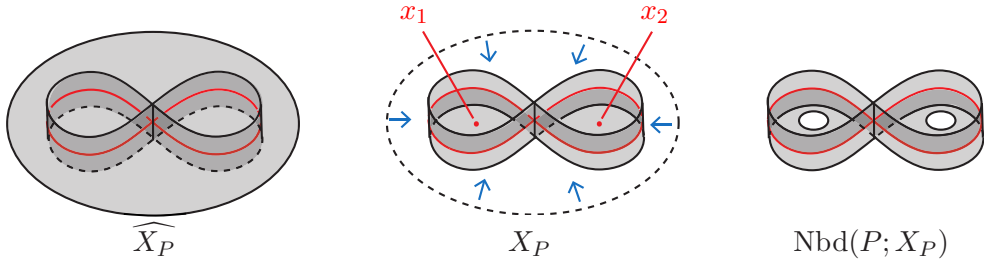
Let $R_1, \ldots, R_n$ be the closures of components of $X_P - \Nbd (P; X_P)$ contained in internal regions of $P$. 
Choose points $x_1 , \dots x_n$ in $R_1, \ldots, R_n$, respectively, and 
set $l_{i} := \pi^{-1}(x_{i})$ ($i=1, \ldots, n$), which is a simple closed curve in $\partial W$. 
Then, the link $L_{X_P}$ is defined to be the union of $\partial X_P$, which is  exactly the link $L_P$ of the divide $P$, and the simple closed curves 
$l_1 , \ldots , l_n$.
Similarly, the link $L_{\widehat{X_P}}$ is defined to be the union of $\partial \widehat{X_P}$ and $l_1 , \ldots , l_n$.
Thus, by definition, we have obvious inclusions $L_P \subset L_{X_P} \subset L_{\widehat{X_P}}$. 
Throughout the paper, we set $M_P := \pi^{-1}(\Nbd(P ; X_P))$. 
The exterior $E(L_P)$ of the link $L_P$ can then be decomposed as 
\[
E(L_{P}) = M_P  \cup \left( \bigcup_{i=1}^n \pi^{-1}(R_i) \right) .
\]
Thus, by identifying each $\pi^{-1}(R_i)$ with a regular neighborhood of $l_i$ in $\partial W$, 
we can think of $E(L_{X_P})$ as $M_P$. 
Figure \ref{figure:divide_l1_l2} shows the links $L_P$, $L_{X_P}$ and $L_{\widehat{X_P}}$ for the divide in Example \ref{ex:divide}. 
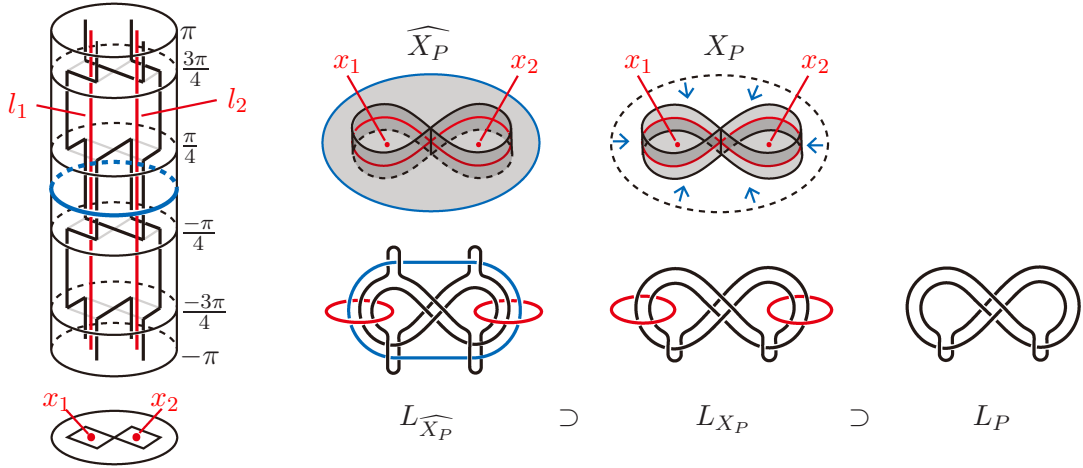
\begin{figure}[htbp]
\centering\includegraphics[width=14cm]{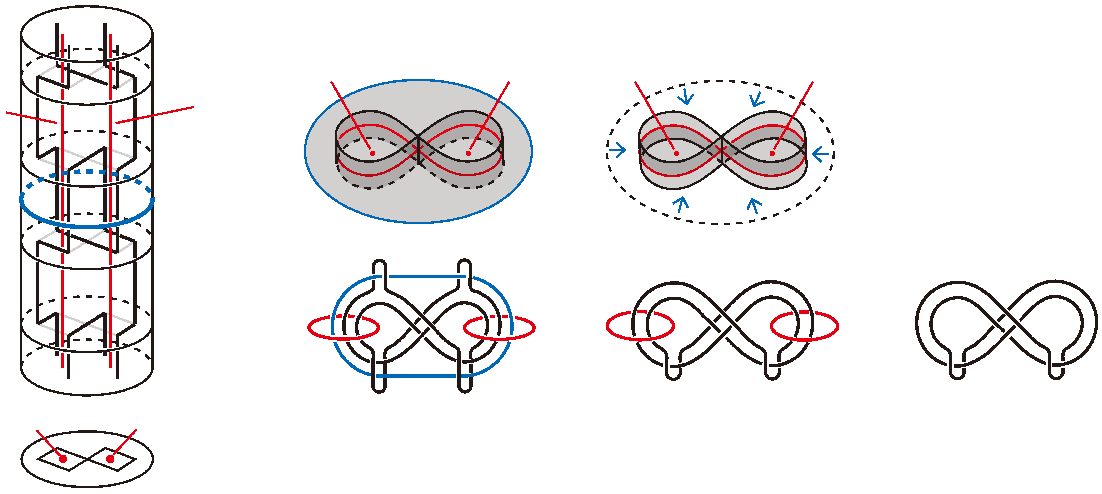}
\begin{picture}(400,0)(0,0)
\put(57,176){$\pi$}
\put(57,162){$\frac{3 \pi}{4}$}
\put(57,133){$\frac{\pi}{4}$}
\put(57,100){$\frac{- \pi}{4}$}
\put(57,70){$\frac{- 3 \pi}{4}$}
\put(57,53){$-\pi$}

\put(5,37){\color{red} $x_1$}
\put(46,37){\color{red} $x_2$}

\put(-8,148){\color{red} $l_1$}
\put(75,150){\color{red} $l_2$}

\put(115,165){\color{red} $x_1$}
\put(182,165){\color{red} $x_2$}

\put(225,165){\color{red} $x_1$}
\put(292,165){\color{red} $x_2$}

\put(143,170){$\widehat{X_P}$}
\put(255,170){$X_P$}

\put(140,30){$L_{\widehat{X_P}}$}
\put(252,30){$L_{X_P}$}
\put(357,30){$L_{P}$}
\put(200,30){$\supset$}
\put(310,30){$\supset$}
\end{picture}
\caption{The links $L_P$, $L_{X_P}$ and $L_{\widehat{X_P}}$.}
\label{figure:divide_l1_l2}
\end{figure}

A regular neighborhood of a vertex of $X_{P}$ 
is that of a vertex of $\widehat{X_{P}}$ minus the external regions.
Thus, it is homeomorphic to one of the six types, 
Types $1, 2,\ldots, 6$, of local models shown in Figure \ref{figure:local_models_of_XP_gen}, 
depending on the configuration of the excluded region, where 
the local coordinates around the vertices are $(x_1, x_2, u_1, u_2)$. 
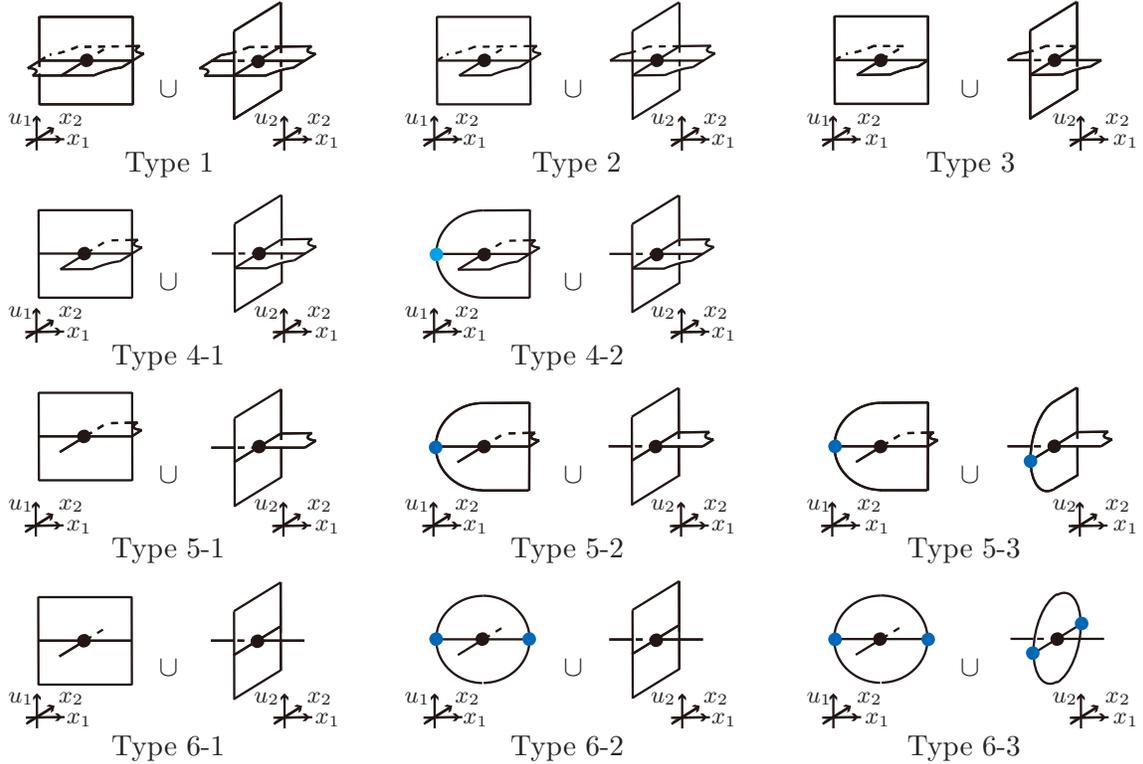
\begin{figure}[htbp]
\centering\includegraphics[width=14.5cm]{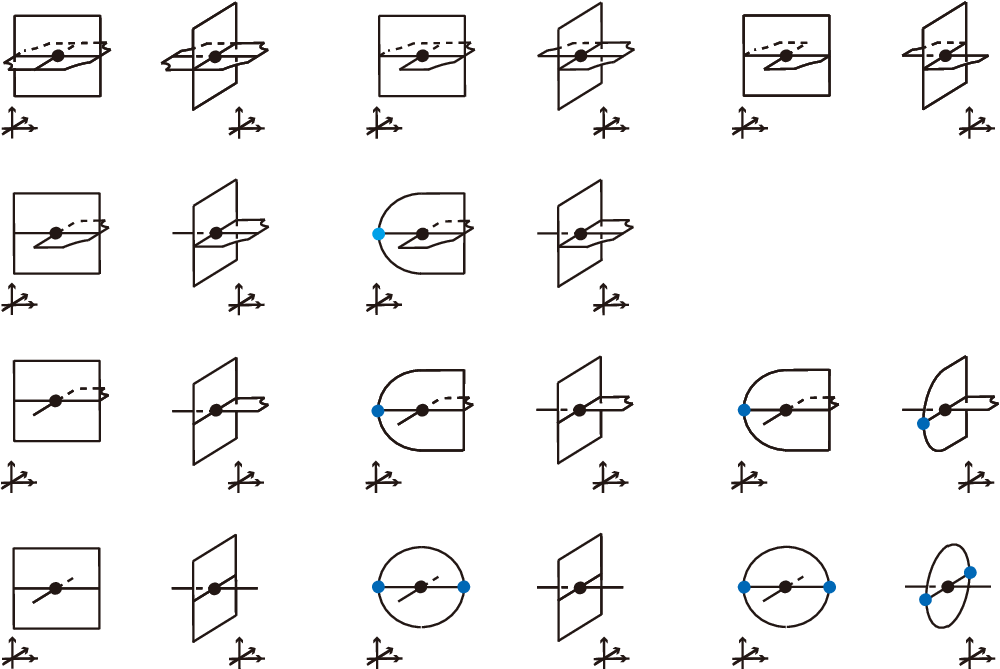}
\begin{picture}(400,0)(0,0)
\put(45,252){$\cup$}
\put(45,179){$\cup$}
\put(45,106){$\cup$}
\put(45,33){$\cup$}

\put(198,252){$\cup$}
\put(198,179){$\cup$}
\put(198,106){$\cup$}
\put(198,33){$\cup$}

\put(348,252){$\cup$}
\put(348,106){$\cup$}
\put(348,33){$\cup$}

\put(10,234){\footnotesize $x_1$}
\put(7,242){\footnotesize $x_2$}
\put(-12,242){\footnotesize $u_1$}
\put(104,234){\footnotesize $x_1$}
\put(101,242){\footnotesize $x_2$}
\put(81,242){\footnotesize $u_2$}

\put(160,234){\footnotesize $x_1$}
\put(157,242){\footnotesize $x_2$}
\put(138,242){\footnotesize $u_1$}
\put(254,234){\footnotesize $x_1$}
\put(251,242){\footnotesize $x_2$}
\put(231,242){\footnotesize $u_2$}

\put(312,234){\footnotesize $x_1$}
\put(309,242){\footnotesize $x_2$}
\put(290,242){\footnotesize $u_1$}
\put(406,234){\footnotesize $x_1$}
\put(403,242){\footnotesize $x_2$}
\put(383,242){\footnotesize $u_2$}

\put(10,161){\footnotesize $x_1$}
\put(7,169){\footnotesize $x_2$}
\put(-12,169){\footnotesize $u_1$}
\put(104,161){\footnotesize $x_1$}
\put(101,169){\footnotesize $x_2$}
\put(81,169){\footnotesize $u_2$}

\put(160,161){\footnotesize $x_1$}
\put(157,169){\footnotesize $x_2$}
\put(138,169){\footnotesize $u_1$}
\put(254,161){\footnotesize $x_1$}
\put(251,169){\footnotesize $x_2$}
\put(231,169){\footnotesize $u_2$}

\put(10,88){\footnotesize $x_1$}
\put(7,96){\footnotesize $x_2$}
\put(-12,96){\footnotesize $u_1$}
\put(104,88){\footnotesize $x_1$}
\put(101,96){\footnotesize $x_2$}
\put(81,96){\footnotesize $u_2$}

\put(160,88){\footnotesize $x_1$}
\put(157,96){\footnotesize $x_2$}
\put(138,96){\footnotesize $u_1$}
\put(254,88){\footnotesize $x_1$}
\put(251,96){\footnotesize $x_2$}
\put(231,96){\footnotesize $u_2$}

\put(312,88){\footnotesize $x_1$}
\put(309,96){\footnotesize $x_2$}
\put(290,96){\footnotesize $u_1$}
\put(406,88){\footnotesize $x_1$}
\put(403,96){\footnotesize $x_2$}
\put(383,96){\footnotesize $u_2$}

\put(10,15){\footnotesize $x_1$}
\put(7,23){\footnotesize $x_2$}
\put(-12,23){\footnotesize $u_1$}
\put(104,15){\footnotesize $x_1$}
\put(101,23){\footnotesize $x_2$}
\put(81,23){\footnotesize $u_2$}

\put(160,15){\footnotesize $x_1$}
\put(157,23){\footnotesize $x_2$}
\put(138,23){\footnotesize $u_1$}
\put(254,15){\footnotesize $x_1$}
\put(251,23){\footnotesize $x_2$}
\put(231,23){\footnotesize $u_2$}

\put(312,15){\footnotesize $x_1$}
\put(309,23){\footnotesize $x_2$}
\put(290,23){\footnotesize $u_1$}
\put(406,15){\footnotesize $x_1$}
\put(403,23){\footnotesize $x_2$}
\put(383,23){\footnotesize $u_2$}

\put(32,224){Type $1$}
\put(186,224){Type $2$}
\put(335,224){Type $3$}
\put(27,151){Type $4$-$1$}
\put(178,151){Type $4$-$2$}
\put(27,78){Type $5$-$1$}
\put(178,78){Type $5$-$2$}
\put(328,78){Type $5$-$3$}
\put(27,3){Type $6$-$1$}
\put(178,3){Type $6$-$2$}
\put(328,3){Type $6$-$3$}
\end{picture}
\caption{The local models of $X_P$. The points on the boundaries colored blue indicate the points corresponding to endpoints of $P$.}
\label{figure:local_models_of_XP_gen}
\end{figure}
Vertices of Types 4, 5, 6 are classified into subclasses as follows. 
Vertices of Types 4-1, 5-1, 6-1 are not connected to endpoints of $P$ by edges. 
Vertices of Types 4-2, 5-2 are connected to a single endpoint of $P$ by an edge. 
Vertices of Types 5-3, 6-2 are connected to two endpoints of $P$ by edges. 
Vertices of Types 6-3 are connected to four endpoints of $P$ by edges. 
In Figure \ref{figure:local_models_of_XP_gen}, endpoints of $P$ are colored in blue. 
As indicated in the figure, for the vertices connected to endpoints of $P$ by edges, that is, for vertices of Types 
4-2, 5-2, 5-3, 6-2, 6-3, we think of neighborhoods of the vertices and the edges to the endpoints as 
local models of them.

A regular neighborhood of a vertex of $X_P$ is homeomorphic to one of the  six types, 
Types 1--6, of local models shown in Figure \ref{figure:local_models_of_XP_gen}. 
A regular neighborhood $\Nbd(P ; X_{P})$ of the divide $P$ is then decomposed into pieces each homeomorphic to 
one of the models. 
Let $C$ be a piece of the decomposition of $\Nbd(P ; X_{P})$. 
We define the following terms for the neighborhood $C$ and its $3$-thickening $M_{C}$. 
\begin{definition}

A connected component of the scars in $\partial C$ of the decomposition of $\Nbd(P;X_{P})$ 
is called an \emph{gluing part} of $C$, and the closure of a connected component of the remaining part $\partial C$ is called a \emph{cusp part} of $C$. See Figure \ref{figure:gluing_cusp}.
Then, connected components of the subsets of $\partial M_C$ corresponding to the gluing part and the cusp part of $C$ is called 
\emph{gluing faces} and \emph{cusp faces}, respectively.
The closure of a connected component obtained by excluding all gluing and cusp faces from $\partial M_{C}$ is called 
a \emph{mirror face} of $M_{C}$. See Figure \ref{figure:type1_glued_faces_and_cusps}.
\end{definition}

\label{figure:gluing_cusp}
\begin{figure}[htbp]
\centering\includegraphics[width=6cm]{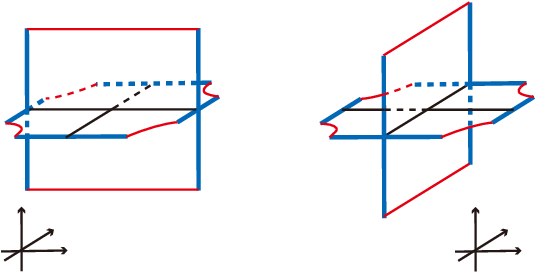}
\begin{picture}(400,0)(0,0)
\put(198,60){$\cup$}
\put(140,17){\footnotesize $x_1$}
\put(134,26){\footnotesize $x_2$}
\put(111,26){\footnotesize $u_1$}
\put(284,17){\footnotesize $x_1$}
\put(279,26){\footnotesize $x_2$}
\put(255,26){\footnotesize $u_2$}
\end{picture}
\caption{Gluing parts (blue, thick) and cusp parts (red, normal).}
\label{figure:gluing_cusp}
\end{figure}
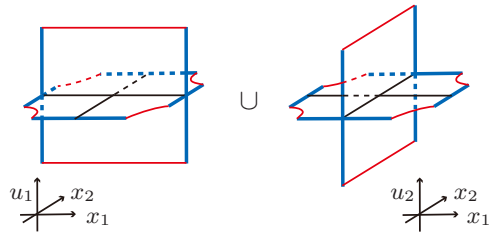
\begin{figure}[htbp]
\centering\includegraphics[width=8cm]{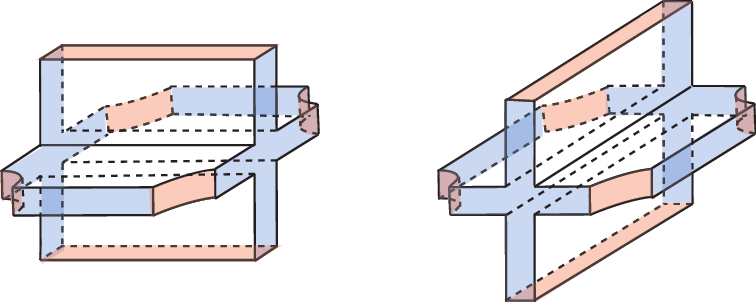}
\begin{picture}(400,0)(0,0)
\put(120,0){$M_1$}
\put(260,0){$M_2$}
\end{picture}
\caption{Gluing faces (blue), cusp faces (red) and mirror faces (white).}
\label{figure:type1_glued_faces_and_cusps}
\end{figure}
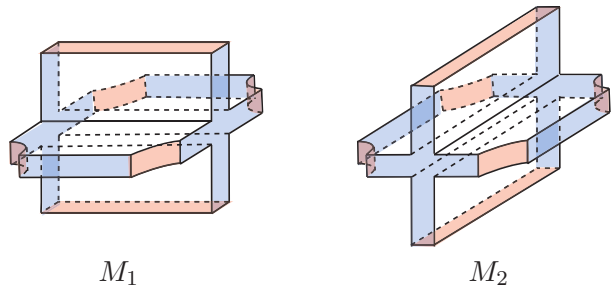

\subsection{Ideal polyhedral decomposition of the 3-manifold $\boldsymbol{ \Int  M_P }$}
\label{subsec:Ideal polyhedral decomposition of the 3-manifold}

In this subsection, we always assume that divides are connected.

\begin{definition}
A divide $P$ is said to be  \textit{prime} if the link $L_{P}$ of $P$ is prime. 
\end{definition}

The following proposition shows that the local models of regular neighborhoods of vertices of $X_P$ for a prime divide $P$ are quite restrictive.

\begin{proposition}
Let $P$ be a divide. 
If $P$ is prime, then each vertex of $X_P$ has a regular neighborhood homeomorphic to one of the six models shown in 
Figure $\ref{figure:local_models_of_XP_prime}$, 
where the blue points in the figure show the endpoints of the divide $P$. 
\end{proposition}
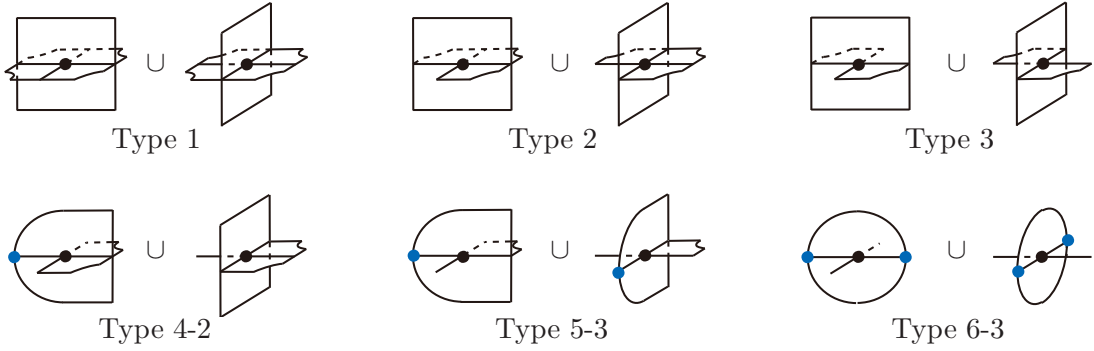
\begin{figure}[htbp]
\centering\includegraphics[width=14.5cm]{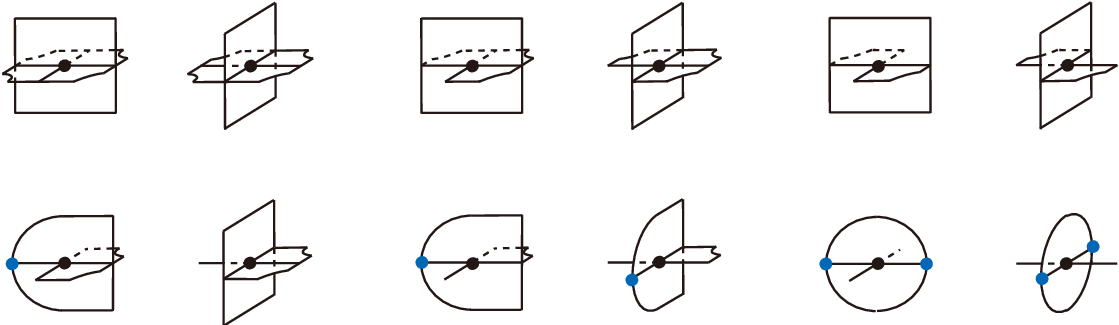}
\begin{picture}(400,0)(0,0)
\put(48,106){$\cup$}
\put(201,106){$\cup$}
\put(351,106){$\cup$}
\put(48,35){$\cup$}
\put(201,35){$\cup$}
\put(351,35){$\cup$}

\put(35,77){Type $1$}
\put(185,77){Type $2$}
\put(335,77){Type $3$}
\put(30,5){Type $4$-$2$}
\put(180,5){Type $5$-$3$}
\put(330,5){Type $6$-$3$}
\end{picture}
\caption{The local models of $X_P$ for a prime divide $P$.}
\label{figure:local_models_of_XP_prime}
\end{figure}
\begin{proof}
Let $P$ be a prime divide.
Suppose first that there exists an edge $E$ whose endpoints are double points of the divide $P$.
In this case, at least one of the two regions of $P$ that touches the edge $E$ must be an internal region of $P$.
Indeed, if both of them are exterior regions, there exists a simple arc $l$ properly embedded in $D$ so that 
$l$ intersects $P$ once and transversely at a point in $E$. 
This arc $l$ decomposes $D$ into two $2$-disks $D_1$ and $D_2$, and $P$ is decomposed into two divides 
$P_1 \subset D_1$, $P_2 \subset D_2$ accordingly.
Then, we have a connected sum decomposition $L_{ P} = L_{P_{1}} \# L_{P_{2}}$ of $L_P$ by the 
$2$-sphere $\{(x, u) \in \partial W \mid x \in l\}$ in $\partial W$. 
From the construction, 
each of the two divides $P_1$ and $P_2$ has at least one double point, which implies that neither $L_{P_1}$ nor $L_{P_2}$  is the trivial knot. 
This contradicts the assumption that $L_P$ is prime. 
The above shows that any edge of $P$ connecting double points of $P$ always touches an inner region of $P$.
Of the eleven (sub)types of models in Figure \ref{figure:local_models_of_XP_gen}, 
those with such edges are of Types 1, 2, 3, 4-2 and 5-3 shown in Figure \ref{figure:local_models_of_XP_prime}.
Therefore, in this case, each vertex of $X_P$ has a regular neighborhood homeomorphic to one of these $5$ types.

Next, suppose that there does not exist an edge whose endpoints are double points of the divide $P$.
Then, any edge of $P$ connects a double point of $P$ and an endpoint of $P$. 
Of the local models in Figure \ref{figure:local_models_of_XP_gen}, 
those with such edges are only of Types $6$-$3$. 
Therefore, in this case, each vertex of $X_P$ has a regular neighborhood homeomorphic to the model of Type $6$-$3$.
\end{proof}

\begin{definition}
Let c be a double point of a prime divide $P$. 
We say that $c$ is of Types $1$, $2$, $3$, $4$-$2$, $5$-$3$ and $6$-$3$ if 
a regular neighborhood $\Nbd (c; X_P)$ of $c$ is homeomorphic to the model of 
Types $1$, $2$, $3$, $4$-$2$, $5$-$3$ and $6$-$3$, respectively. 
\end{definition}

\begin{theorem}
\label{theo:main}
Let $P \subset D$ be a connected prime divide with at least one double point. 
If $P$ has a double point of Type $6$-$3$, then $L_P$ is the Hopf link. 
Otherwise, let $n_1$, $n_2$, $n_3$, $n_4$ and $n_5$ be the number of its double points of 
Types $1$, $2$, $3$, $4$-$2$ and $5$-$3$, respectively.  
Then $\Int M_P$ is a hyperbolic $3$-manifold of volume 
\[ 10 n_3 v_{\mathrm{tet}} + (4n_1 + 2n_4 + n_5) v_{\mathrm{oct}} + n_2 v_{\mathrm{cuboct}}. \]
Here, $v_{\mathrm{tet}}$, $v_{\mathrm{oct}}$ and $v_{\mathrm{cuboct}}$ are 
the volumes of ideal hyperbolic regular tetrahedron, octahedron and cuboctahedron, respectively. 
\end{theorem}

We remark that if a divide $P$ has a double point $c$ of Type $6$-$3$,  it follows from the connectivity of $P$ that 
$P$ is of the $X$-shape and $c$ is its unique double point. 
In this case, $\Int M_P$ is homeomorphic to the Hopf link complement, 
which does not admits a hyperbolic structure. 

In the following, we always assume that a divide $P$ is prime, and there are no vertices of Type $6$-$3$ in $X_P$. 
In order to prove Theorem \ref{theo:main}, we here define several terminology and symbols.


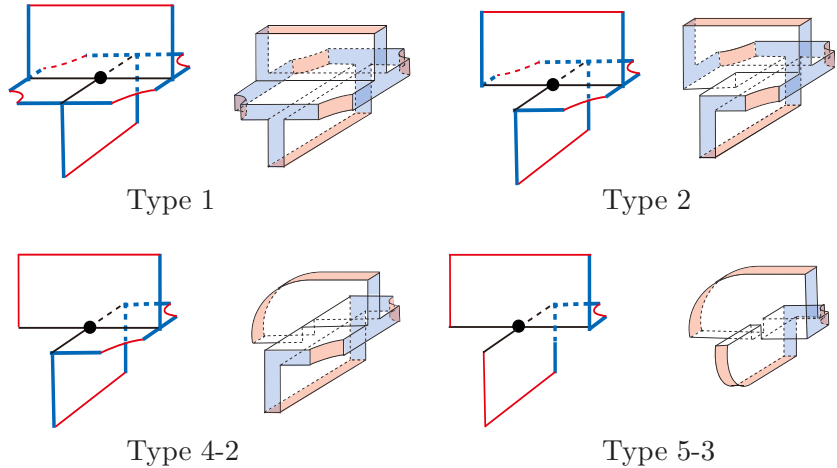
\begin{figure}[htbp]
\centering\includegraphics[width=11cm]{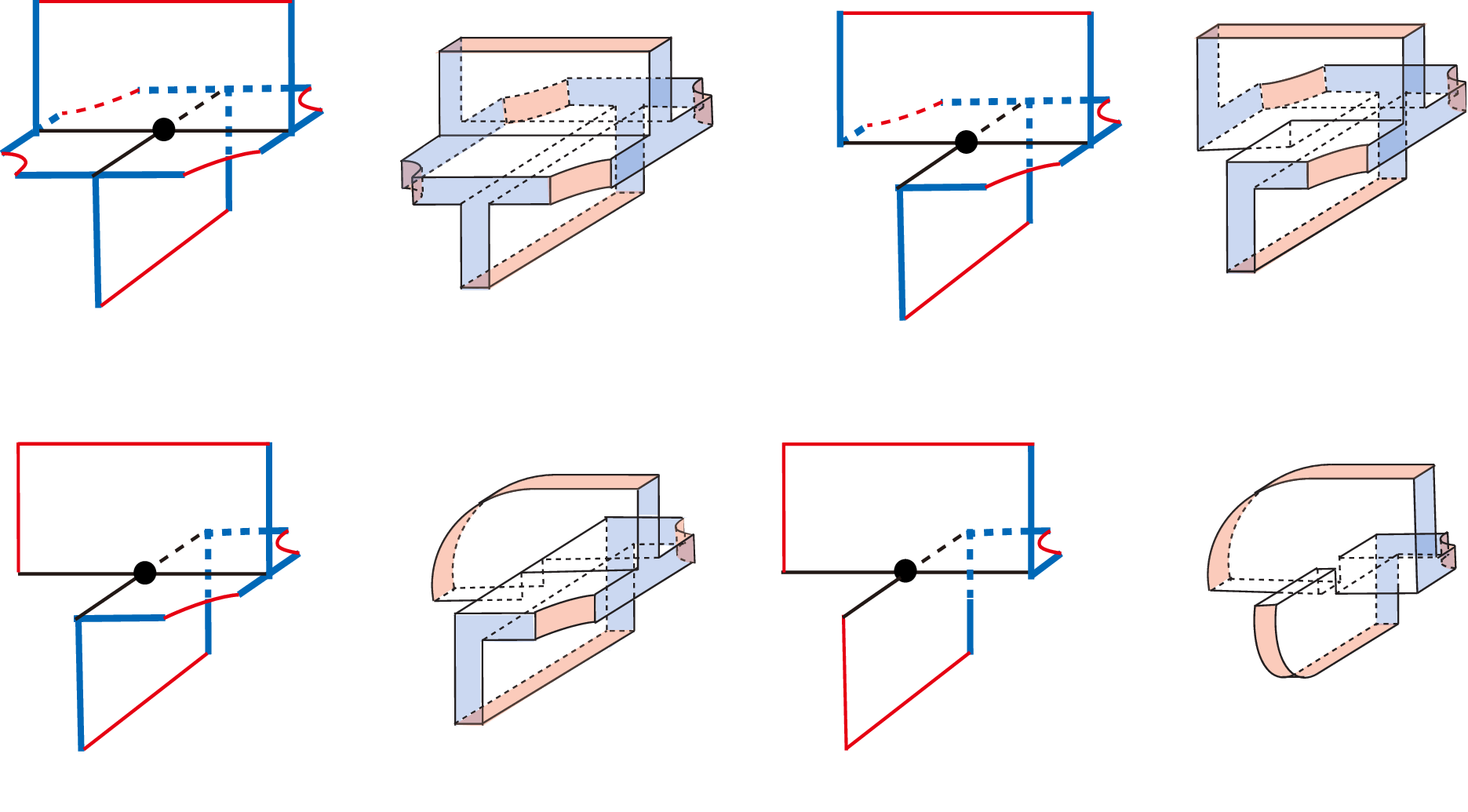}
\caption{The quotient $C / \iota$ and its $3$-thickening $M_{C/\iota}$.}
\begin{picture}(400,0)(0,0)
\put(90,130){Type $1$}
\put(270,130){Type $2$}
\put(90,35){Type $4$-$2$}
\put(270,35){Type $5$-$3$}
\end{picture}
\label{figure:C_over_tau}
\end{figure}

Let $c$ be a double point of $P$, which is not of Type $3$. 
A regular neighborhood $\Nbd(c; X_P) \subset \{(x_{1},x_{2},u_{1},u_{2}) \in \RR^{4}\}$ of $c$ is 
one of the models of Types $1$, $2$, $4$-$2$ and $5$-$3$ shown in Figure \ref{figure:local_models_of_XP_prime}. 
Let $C / \iota$ be the quotient of $C$ by the involution $\iota : W \to W, (x,u) \to (x, -u)$ of the $4$-ball $W$. 
Then the quotient $C / \iota$ and its $3$-thickening $M_{C / \iota}$ for each Type is as shown in Figure \ref{figure:C_over_tau}.
The gluing, cusp, and mirror faces of $M_{C / \iota}$ are naturally defined from those of $M_{C}$. 
By the same idea of Turaev's reconstruction explained in the outlined proof of Theorem \ref{prop:2v8+10v4}, 
the preimage $\pi^{-1}(C) / \iota$ is obtained by gluing two copies $M_{C/\iota}^{+}$ and $M_{ C/\iota}^{-}$ of $M_{C / \iota}$ 
along the corresponding mirror faces.
For each $\epsilon \in  \{+,-\}$, let $M_{C/\iota}^{\epsilon +}$, $M_{C/\iota}^{\epsilon -}$ be the two lifts of 
$M_{C/\iota}^{\epsilon}$ in $\pi^{-1}(C)$. 
Then, the decomposition $\pi^{-1}(C) / \iota = M_{C/\iota}^{+} \cup M_{C/\iota}^{-}$ induces 
a decomposition of $\pi^{-1}(C)$ into four copies of $M_{C / \iota}$, where they are glued along 
their mirror faces in a natural way. 

By the way, we can regard the 3-thickening $M_{C / \iota}$ as a truncated hyperbolic polyhedron as shown in Figure 
\ref{figure:identify_3thickening_polyhedron}. 
In fact, if $c$ is of Type $1$, $M_{C / \iota}$ is nothing but a truncated ideal regular hyperbolic octahedron, and 
if $c$ is of Type $2$, $M_{C / \iota}$ is obtained by deforming a truncated ideal regular hyperbolic octahedron so that 
one facet shrinks to be an edge of dihedral angle $\pi/2$. 
Similarly, by applying the deformations shown in Figure \ref{figure:identify_3thickening_polyhedron}, 
we see that if $c$ is of Type $4$-$2$ and $5$-$3$, $M_{C / \iota}$ is a half and a quoter of 
a truncated ideal regular hyperbolic octahedron, respectively. 
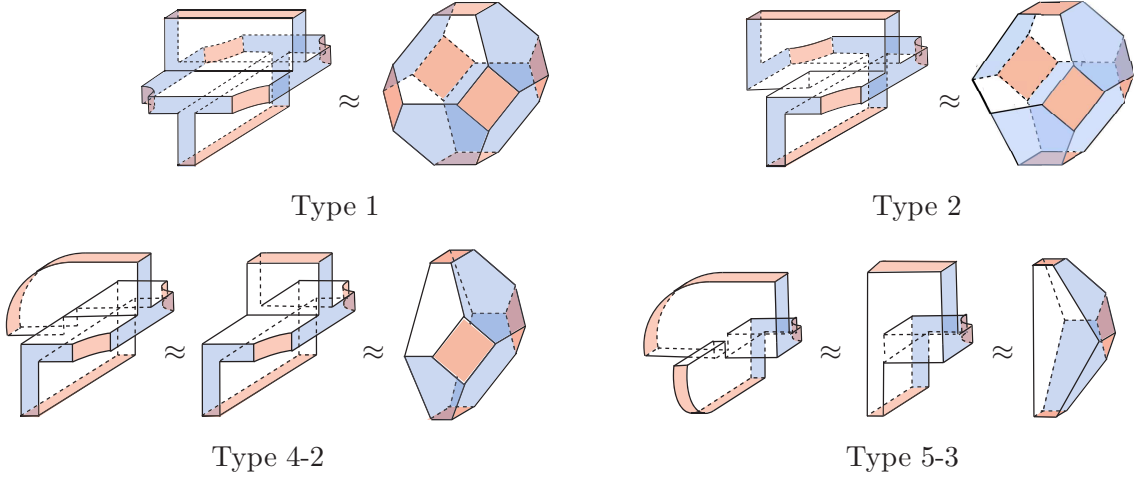
\begin{figure}[htbp]
\centering\includegraphics[width=15cm]{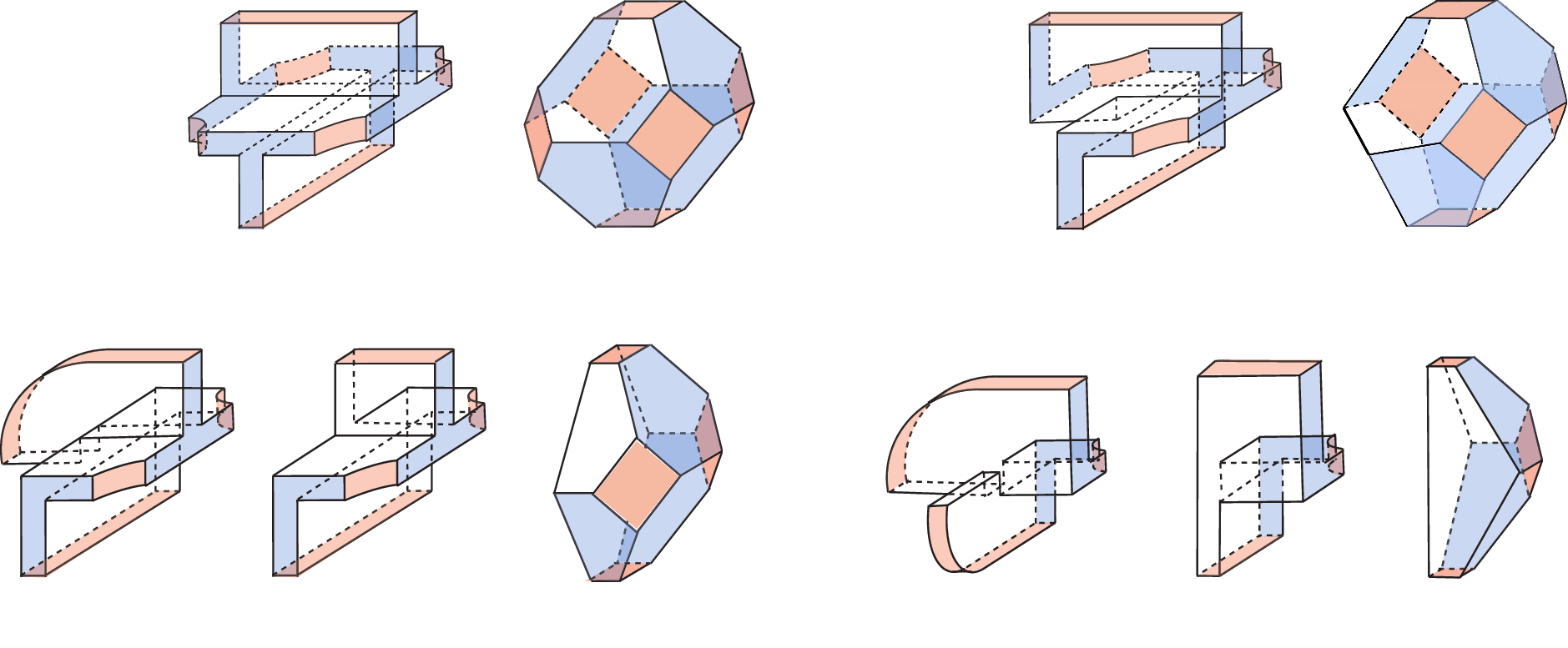}
\caption{The $3$-thickenings $M_{C/\iota}$ as truncated hyperbolic polyhedra.} 
\begin{picture}(400,0)(0,0)
\put(118,170){$\approx$}
\put(345,170){$\approx$}
\put(52,75){$\approx$}
\put(127,75){$\approx$}
\put(300,75){$\approx$}
\put(365,75){$\approx$}

\put(100,130){Type $1$}
\put(320,130){Type $2$}

\put(70,35){Type $4$-$2$}
\put(310,35){Type $5$-$3$}

\end{picture}
\label{figure:identify_3thickening_polyhedron}
\end{figure}

Based on this observation, 
we can show the following Lemmas \ref{lem:type-1}--\ref{lem:type-3}.

\begin{lemma}
\label{lem:type-1}
Let $c$ be a vertex of $X_P$ of Type $1$. 
Then $\pi^{-1}(C)$ is a genus-$4$ handlebody obtained by gluing the faces of four truncated ideal regular hyperbolic octahedra 
as shown in Figure $\ref{figure:type1_decomposition}$. 
This genus-$4$ handlebody has four four-holed spheres on the boundary, 
each of which consists of four faces of these truncated octahedra.
\end{lemma}
\begin{figure}[htbp]
\centering\includegraphics[width=14cm]{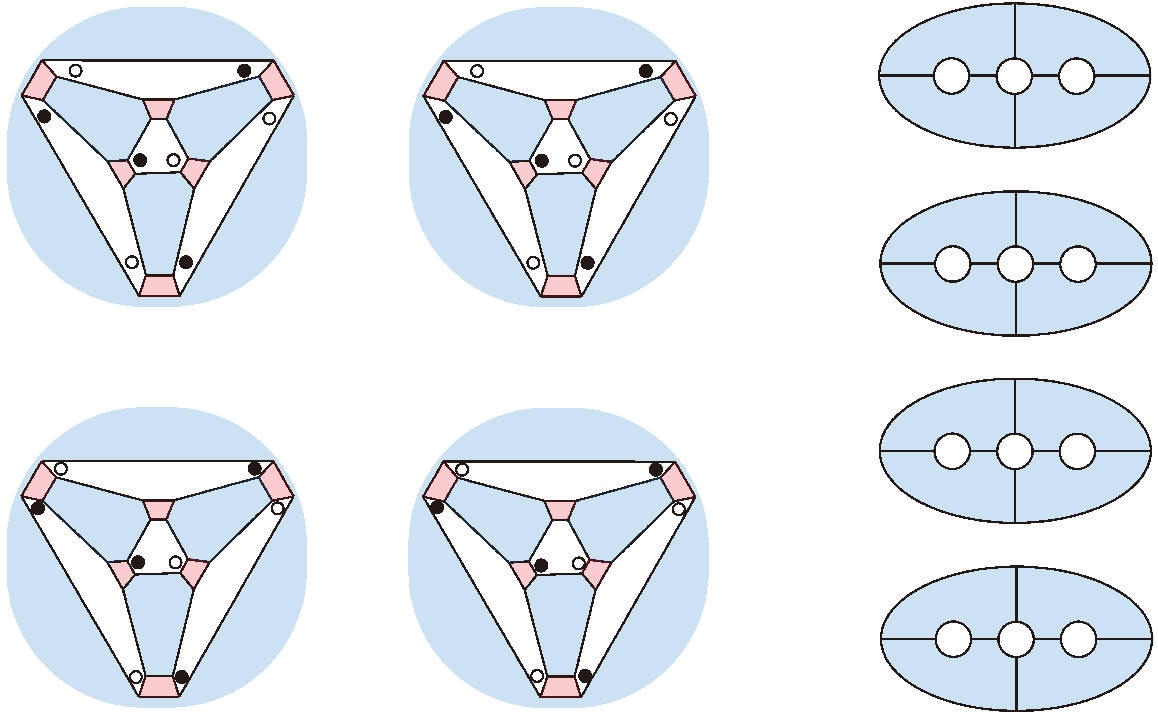}
\begin{picture}(400,0)(0,0)
\put(51,230){$1$}
\put(28,193){$2$}
\put(74,193){$3$}
\put(52,206){$4$}
\put(51,177){\color{blue}$1$}
\put(70,215){\color{blue}$2$}
\put(30,215){\color{blue}$3$}
\put(15,177){\color{blue}$4$}

\put(51,90){$7$}
\put(28,54){$8$}
\put(74,54){$3$}
\put(52,67){$4$}
\put(51,40){\color{blue}$9$}
\put(68,78){\color{blue}$10$}
\put(28,78){\color{blue}11$$}
\put(13,40){\color{blue}$12$}

\put(190,230){$1$}
\put(167,193){$2$}
\put(213,193){$5$}
\put(190,206){$6$}
\put(190,177){\color{blue}$5$}
\put(209,215){\color{blue}$6$}
\put(169,215){\color{blue}$7$}
\put(154,177){\color{blue}$8$}

\put(190,90){$7$}
\put(167,54){$8$}
\put(213,54){$5$}
\put(190,67){$6$}
\put(187,40){\color{blue}$13$}
\put(207,78){\color{blue}$14$}
\put(167,78){\color{blue}15$$}
\put(152,40){\color{blue}$16$}

\put(326,244){\color{blue}$1$}
\put(326,215){\color{blue}$9$}
\put(326,178){\color{blue}$2$}
\put(324,150){\color{blue}$10$}
\put(326,114){\color{blue}$3$}
\put(326,86){\color{blue}$7$}
\put(326,48){\color{blue}$4$}
\put(326,22){\color{blue}$8$}

\put(368,244){\color{blue}$5$}
\put(366,215){\color{blue}$13$}
\put(368,178){\color{blue}$6$}
\put(366,150){\color{blue}$14$}
\put(366,114){\color{blue}$11$}
\put(366,86){\color{blue}$15$}
\put(366,48){\color{blue}$12$}
\put(366,22){\color{blue}$16$}

\end{picture}
\caption{A polyhedral decomposition of $\pi^{-1}(C)$ for a Type $1$ vertex $c$.}
\label{figure:type1_decomposition}
\end{figure}
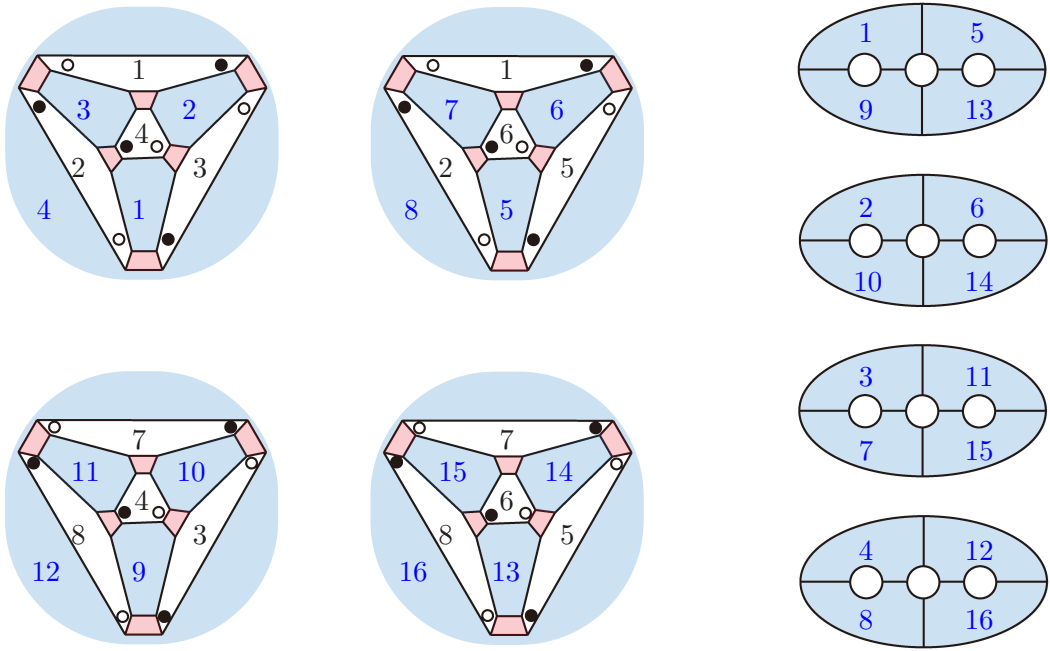
\begin{proof}

In this case, $M_{C/\iota}$ is a truncated ideal regular hyperbolic octahedron shown in Figure \ref{figure:type1_block_and_octahedron}, 
and the decomposition $\pi^{-1}(C) = M_{C / \iota}^{++} \cup M_{C / \iota}^{+-} \cup M_{C / \iota}^{-+} \cup M_{C / \iota}^{--}$ 
is as in Figure~\ref{figure:type1_four_octahedra}. 
\begin{figure}[htbp]
\centering\includegraphics[width=9cm]{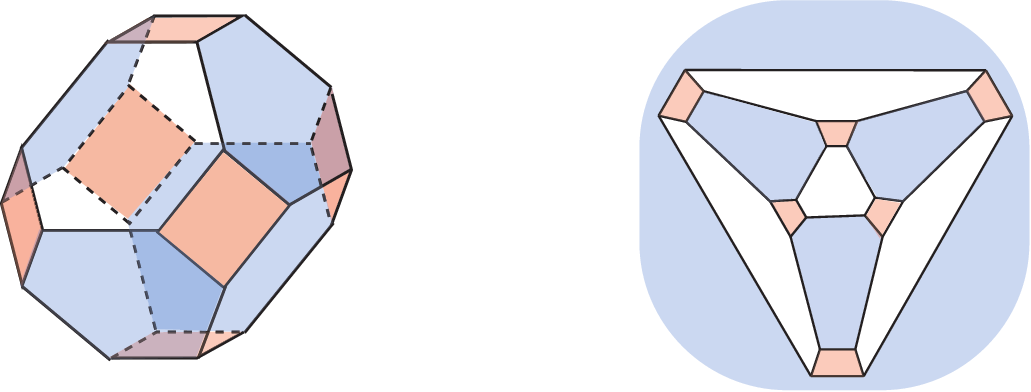}
\begin{picture}(400,0)(0,0)
\end{picture}
\caption{A piece $M_{C/\iota}$ for a Type $1$ vertex.}
\label{figure:type1_block_and_octahedron}
\end{figure}
\begin{figure}[htbp]
\centering\includegraphics[width=8cm]{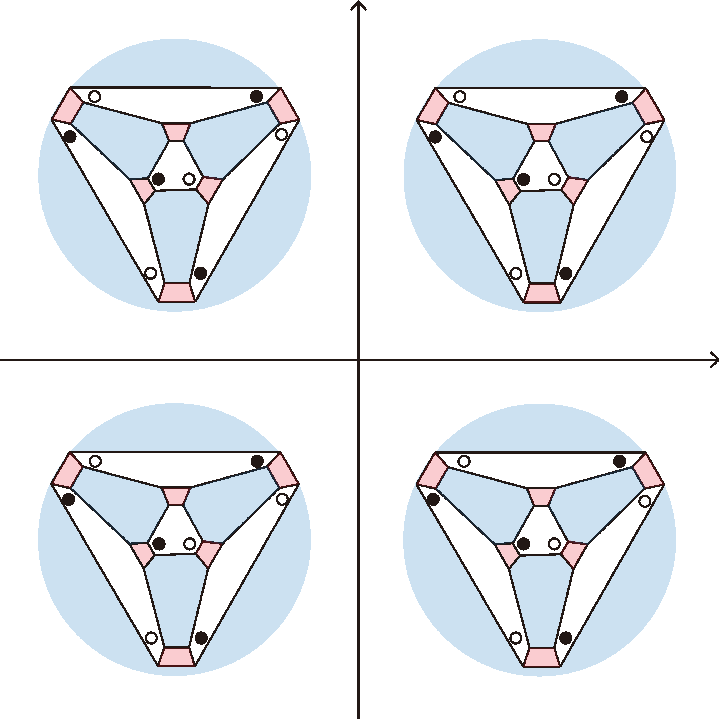}
\begin{picture}(400,0)(0,0)
\put(139,204){$1$}
\put(123,173){$2$}
\put(157,173){$3$}
\put(139,186){$4$}
\put(139,162){\color{blue}$1$}
\put(157,194){\color{blue}$2$}
\put(123,194){\color{blue}$3$}
\put(110,159){\color{blue}$4$}

\put(139,89){$7$}
\put(123,57){$8$}
\put(156,57){$3$}
\put(139,70){$4$}
\put(139,46){\color{blue}$9$}
\put(152,77){\color{blue}$10$}
\put(118,77){\color{blue}11$$}
\put(110,40){\color{blue}$12$}

\put(254,204){$1$}
\put(237,173){$2$}
\put(272,173){$5$}
\put(254,186){$6$}
\put(254,162){\color{blue}$5$}
\put(272,194){\color{blue}$6$}
\put(238,194){\color{blue}$7$}
\put(224,159){\color{blue}$8$}

\put(254,89){$7$}
\put(237,57){$8$}
\put(273,57){$5$}
\put(254,70){$6$}
\put(251,46){\color{blue}$13$}
\put(268,77){\color{blue}$14$}
\put(236,77){\color{blue}15$$}
\put(224,40){\color{blue}$16$}

\put(300,217){$M_{C/\iota}^{++}$}
\put(300,27){$M_{C/\iota}^{+-}$}
\put(70,217){$M_{C/\iota}^{-+}$}
\put(70,27){$M_{C/\iota}^{--}$}

\put(310,116){$u_1$}
\put(205,237){$u_2$}
\end{picture}
\caption{The decomposition $\pi^{-1}(C)=M_{C/\iota}^{++} \cup M_{C/\iota}^{+-} \cup M_{C/\iota}^{-+} \cup M_{C/\iota}^{--}$.}
\label{figure:type1_four_octahedra}
\end{figure}
The resulting 3-manifold $\pi^{-1}(C)$ by the gluing is clearly a genus-$4$ handlebody. 
From the way of gluing, this genus-$4$ handlebody has four four-holed spheres on the boundary, each of which consists 
of four faces of these truncated octahedra as shown on the right in Figure \ref{figure:type1_decomposition}.
\end{proof}

\begin{lemma}
\label{lem:type-2}
Let $c$ be a vertex of $X_P$ of Type $2$. 
Then $\pi^{-1}(C)$ is a genus-$4$ handlebody obtained by gluing the faces of a single truncated ideal regular hyperbolic cuboctahedron 
as shown in Figure $\ref{figure:type2_decomposition}$. 
On the boundary of this genus-$4$ handlebody, there exist 
two three-holed spheres, each of which consists of a single face of this cuboctahedon, 
and two four-holed spheres, 
each of which consists of two faces.  
\end{lemma}
\begin{figure}[htbp]
\centering\includegraphics[width=12cm]{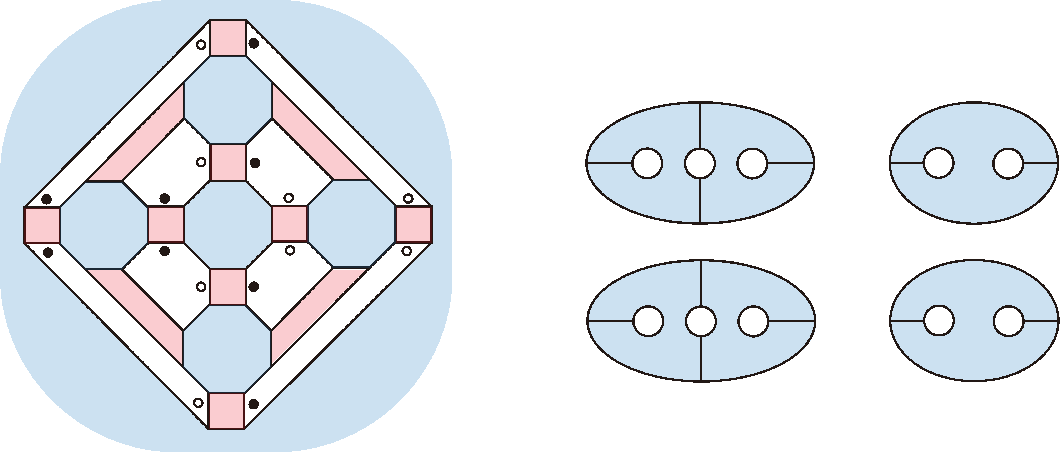}
\begin{picture}(400,0)(0,0)
\put(100,124){\color{blue}$1$}
\put(60,84){\color{blue}$2$}
\put(100,42){\color{blue}$3$}
\put(140,84){\color{blue}$4$}
\put(100,84){\color{blue}$5$}
\put(52,35){\color{blue}$6$}

\put(68,117){$1$}
\put(68,49){$1$}

\put(118,65){$2$}
\put(132,49){$2$}

\put(132,117){$3$}
\put(118,100){$3$}

\put(82,65){$4$}
\put(82,100){$4$}

\put(340,93){\color{blue}$2$}

\put(340,43){\color{blue}$4$}

\put(235,91){\color{blue}$1$}
\put(265,91){\color{blue}$3$}

\put(235,41){\color{blue}$5$}
\put(265,41){\color{blue}$6$}
\end{picture}
\caption{A polyhedral decomposition of $\pi^{-1}(C)$ for a Type $2$ vertex $c$.}
\label{figure:type2_decomposition}
\end{figure}
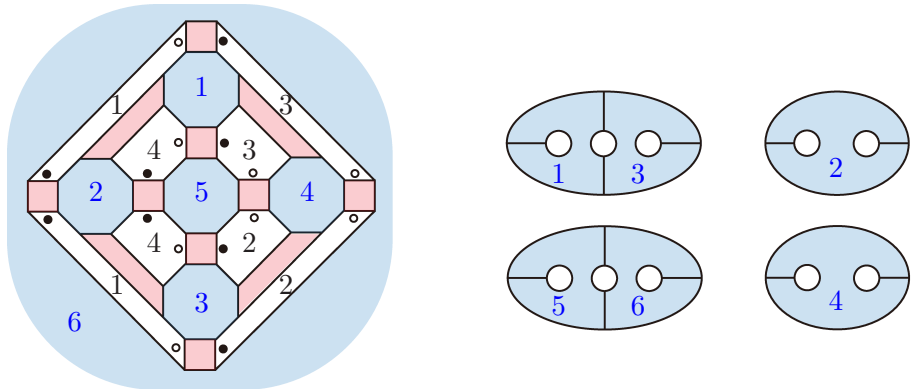
\begin{proof}

In this case, 
$M_{C/\iota}$ is the hyperbolic polyhedron obtained by deforming a truncated ideal regular hyperbolic octahedron so that 
one facet shrinks to be an edge of dihedral angle $\pi/2$, 
and the decomposition $\pi^{-1}(C) = M_{C / \iota}^{++} \cup M_{C / \iota}^{+-} \cup M_{C / \iota}^{-+} \cup M_{C / \iota}^{--}$ 
is as in Figure \ref{figure:type2_four_octahedra}.
\begin{figure}[htbp]
\centering\includegraphics[width=10cm]{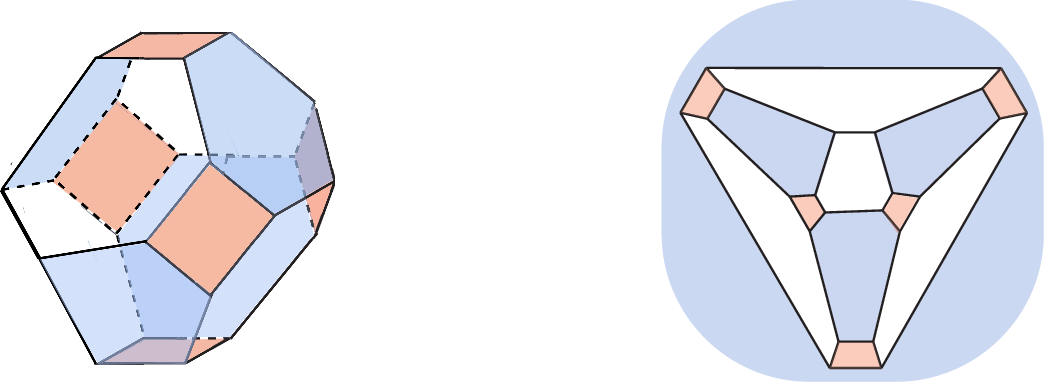}
\begin{picture}(400,0)(0,0)
\end{picture}
\caption{A piece $M_{C/\iota}$ for a Type $2$ vertex.}
\label{figure:type2_block_and_octahedron}
\end{figure}
\begin{figure}[htbp]
\centering\includegraphics[width=8cm]{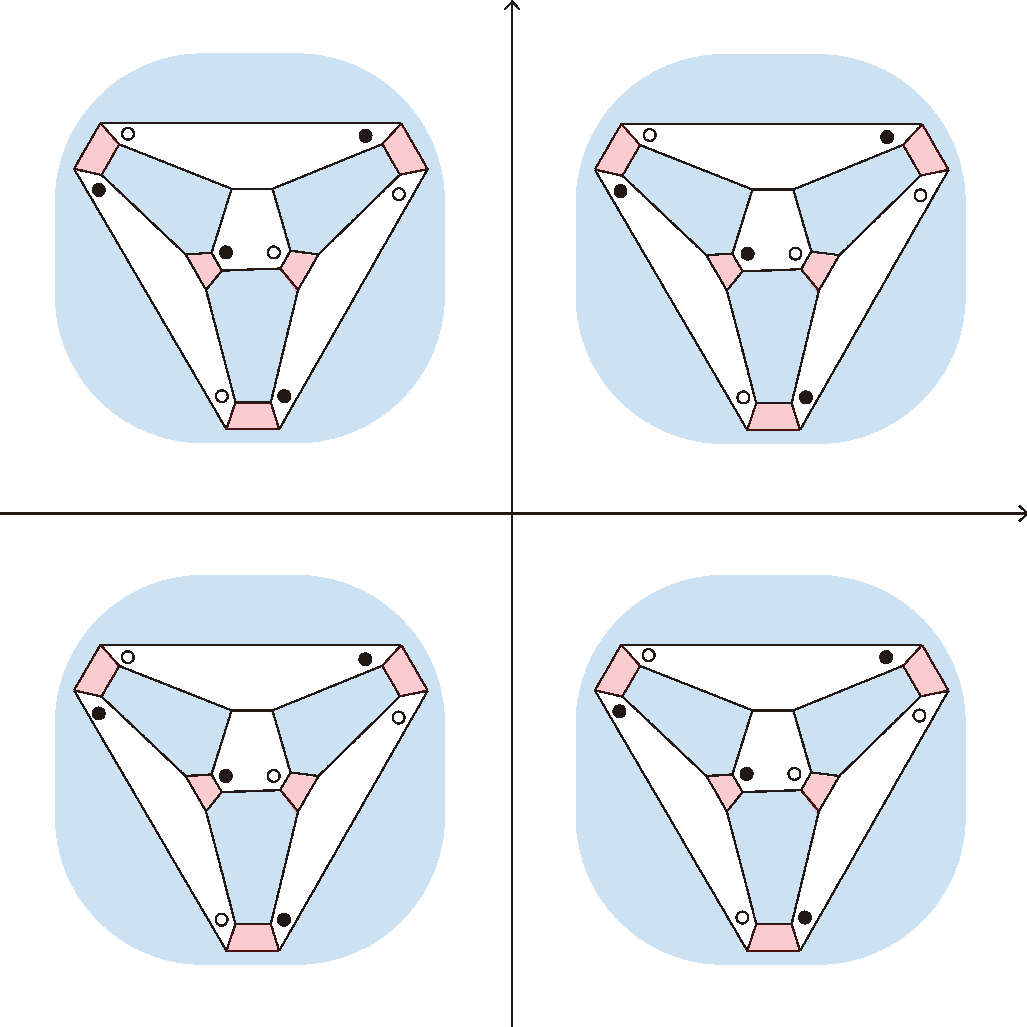}
\begin{picture}(400,0)(0,0)
\put(139,203){$1$}
\put(122,173){$2$}
\put(157,173){$3$}
\put(139,187){$4$}

\put(139,87){$7$}
\put(121,59){$8$}
\put(157,59){$3$}
\put(139,72){$4$}

\put(254,203){$1$}
\put(237,173){$2$}
\put(272,173){$5$}
\put(254,187){$6$}

\put(254,87){$7$}
\put(237,59){$8$}
\put(272,59){$5$}
\put(254,72){$6$}

\put(300,217){$M_{C/\iota}^{++}$}
\put(300,27){$M_{C/\iota}^{+-}$}
\put(70,217){$M_{C/\iota}^{-+}$}
\put(70,27){$M_{C/\iota}^{--}$}

\put(310,116){$u_1$}
\put(205,237){$u_2$}
\end{picture}
\caption{The decomposition $\pi^{-1}(C)=M_{C/\iota}^{++} \cup M_{C/\iota}^{+-} \cup M_{C/\iota}^{-+} \cup M_{C/\iota}^{--}$.}
\label{figure:type2_four_octahedra}
\end{figure}
Glue together the two polyhedra $M_{C/\iota}^{++}$ and $M_{C/\iota}^{-+}$ along the faces assigned $1$, 
and glue together $M_{C/\iota}^{+-}$ and $M_{C/\iota}^{- -}$ along the faces assigned $7$, according to Figure \ref{figure:type2_four_octahedra}. 
Then, we get the two polyhedra $M_{C/\iota}^{++} \cup M_{C/\iota}^{-+}$ and $M_{C/\iota}^{--} \cup M_{C/\iota}^{+-}$ shown in Figure \ref{figure:type2_two_octahedra_glued}. 
\begin{figure}[htbp]
\centering\includegraphics[width=13cm]{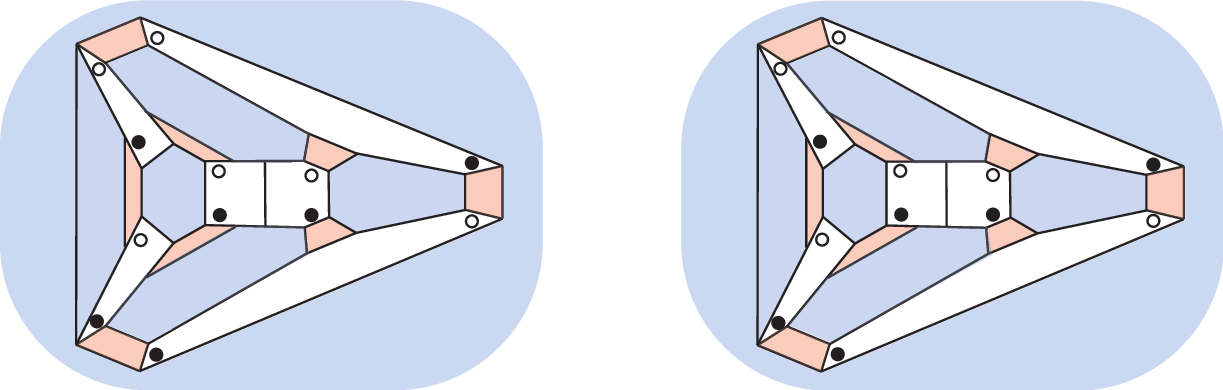}
\begin{picture}(400,0)(0,0)

\put(51,45){$2$}
\put(112,47){$2$}
\put(112,93){$3$}
\put(100,69){$4$}
\put(85,68){$6$}
\put(51,93){$5$}

\put(257,45){$8$}
\put(318,47){$8$}
\put(318,93){$3$}
\put(306,69){$4$}
\put(291,69){$6$}
\put(257,93){$5$}

\end{picture}
\caption{The polyhedra $M_{C/\iota}^{++} \cup M_{C/\iota}^{-+}$ (left) and $M_{C/\iota}^{+-} \cup M_{C/\iota}^{--}$ (right).}
\label{figure:type2_two_octahedra_glued}
\end{figure}
Furthermore, gluing together these two big polyhedra $M_{C/\iota}^{++} \cup M_{C/\iota}^{-+}$ and $M_{C/\iota}^{+-} \cup M_{C/\iota}^{--}$ 
along the faces assigned 4 and 6, which are adjacent each other, we get a single polyhedron shown in Figure \ref{figure:type2_decomposition} 
after renumbering the faces, which is a truncated ideal regular hyperbolic cuboctahedron.

The resulting 3-manifold $\pi^{-1}(C)$ by the gluing is again a genus-$4$ handlebody. 
From the way of gluing, it is easily checked that 
on the boundary of this genus-$4$ handlebody, there exist 
two three-holed spheres, each of which consists of a single face of this cuboctahedron, 
and two four-holed spheres, 
each of which consists of two faces of this cuboctahedron as shown on the right in Figure \ref{figure:type2_decomposition}.  
\end{proof}

\begin{lemma}
\label{lem:type-4}
Let $c$ be a vertex of $X_P$ of Type $4$-$2$. 
Then $\pi^{-1}(C)$ is a genus-$3$ handlebody obtained by gluing the faces of two truncated ideal regular hyperbolic octahedra 
as shown in Figure $\ref{figure:type4_decomposition}$. 
On the boundary of this genus-$3$ handlebody, there exist 
two three-holed spheres, each of which consists of two face of these octahedra, 
and a single four-holed spheres, 
which consists of four faces.  
\end{lemma}
\begin{figure}[htbp]
\centering\includegraphics[width=14cm]{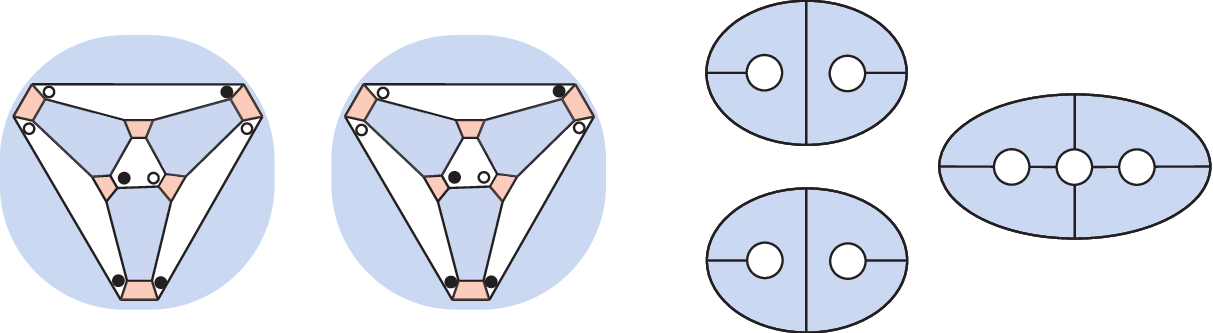}
\begin{picture}(400,0)(0,0)

\put(43,86){$1$}
\put(24,55){$1$}
\put(63,55){$3$}
\put(43,67){$4$}

\put(43,44){\color{blue}$1$}
\put(61,75){\color{blue}$2$}
\put(26,75){\color{blue}$3$}
\put(12,40){\color{blue}$4$}

\put(152,86){$2$}
\put(133,55){$2$}
\put(172,55){$3$}
\put(152,67){$4$}

\put(152,44){\color{blue}$5$}
\put(170,75){\color{blue}$6$}
\put(135,75){\color{blue}$7$}
\put(121,40){\color{blue}$8$}

\put(247,82){\color{blue}$3$}
\put(277,82){\color{blue}$7$}

\put(247,20){\color{blue}$4$}
\put(277,20){\color{blue}$8$}

\put(330,51){\color{blue}$1$}
\put(372,51){\color{blue}$2$}

\put(330,78){\color{blue}$5$}
\put(372,78){\color{blue}$6$}
\end{picture}
\caption{A polyhedral decomposition of $\pi^{-1}(C)$ for a Type $4$-$2$ vertex $c$.}
\label{figure:type4_decomposition}
\end{figure}
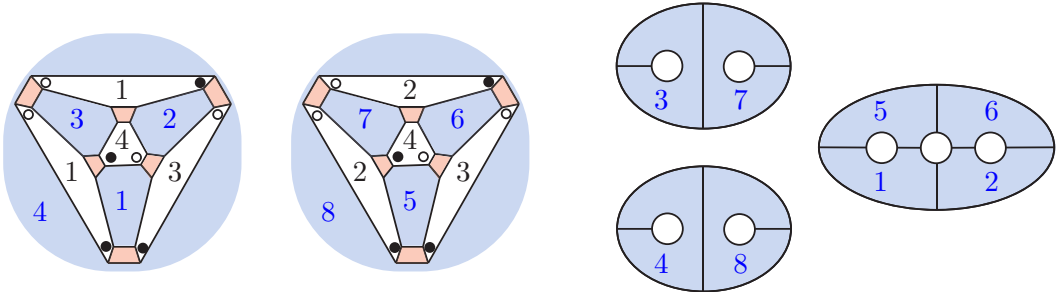
\begin{proof}

In this case, $M_{C/\iota}$ is the half of a truncated ideal regular hyperbolic octahedron shown in Figure \ref{figure:identify_3thickening_polyhedron}, 
and the decomposition $\pi^{-1}(C) = M_{C / \iota}^{++} \cup M_{C / \iota}^{+-} \cup M_{C / \iota}^{-+} \cup M_{C / \iota}^{--}$ 
is as in Figure \ref{figure:type4_four_octahedra}.
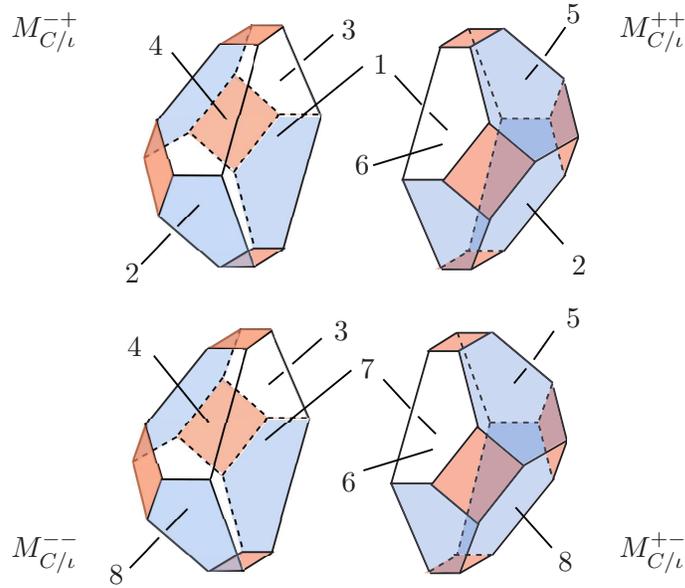
\begin{figure}[htbp]
\centering\includegraphics[width=6cm]{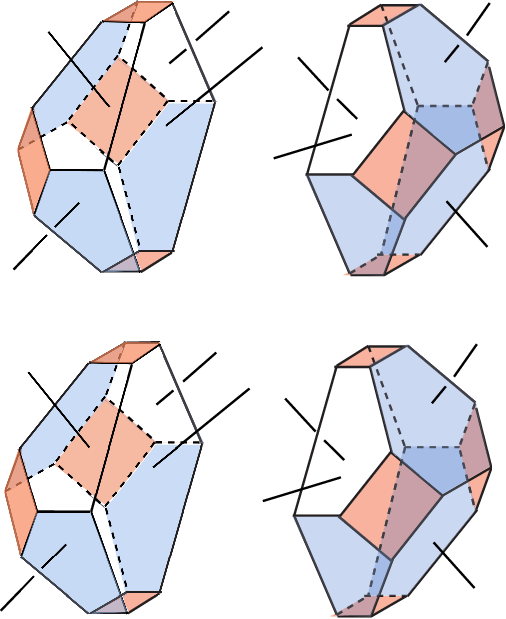}
\begin{picture}(400,0)(0,0)
\put(207,203){$1$}
\put(113,124){$2$}
\put(195,216){$3$}
\put(122,210){$4$}
\put(202,88){$7$}
\put(107,10){$8$}
\put(192,102){$3$}
\put(114,96){$4$}

\put(282,127){$2$}
\put(200,165){$6$}
\put(280,222){$5$}

\put(195,45){$6$}
\put(280,107){$5$}
\put(277,15){$8$}

\put(300,217){$M_{C/\iota}^{++}$}
\put(300,20){$M_{C/\iota}^{+-}$}
\put(70,217){$M_{C/\iota}^{-+}$}
\put(70,20){$M_{C/\iota}^{--}$}

\end{picture}

\caption{The decomposition $\pi^{-1}(C)=M_{C/\iota}^{++} \cup M_{C/\iota}^{+-} \cup M_{C/\iota}^{-+} \cup M_{C/\iota}^{--}$.}
\label{figure:type4_four_octahedra}
\end{figure}
Glue together the two polyhedra $M_{C/\iota}^{++}$ and $M_{C/\iota}^{-+}$ along the faces assigned $1$, 
and glue together $M_{C/\iota}^{--}$ and $M_{C/\iota}^{+ -}$ along the faces assigned $7$, according to Figure \ref{figure:type4_four_octahedra}. 
Then, both of $M_{C/\iota}^{++} \cup M_{C/\iota}^{-+}$ and $M_{C/\iota}^{+-} \cup M_{C/\iota}^{--}$ are truncated ideal regular hyperbolic octahedra as shown in Figure \ref{figure:type4_two_octahedra_glued}. 
\begin{figure}[htbp]
\centering\includegraphics[width=10cm]{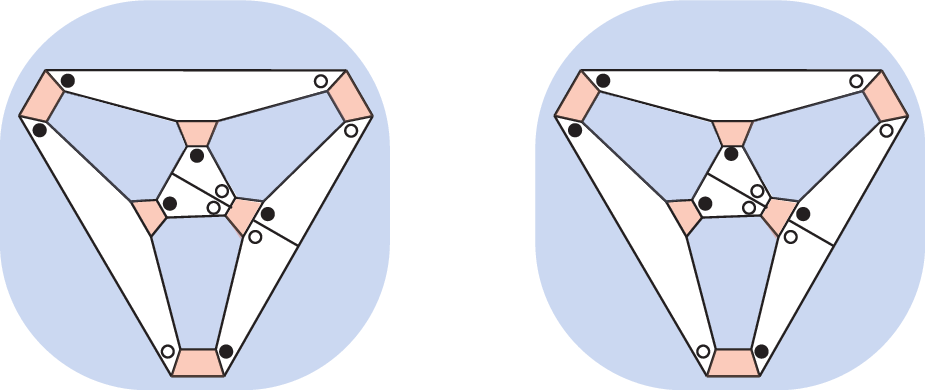}
\begin{picture}(400,0)(0,0)

\put(115,101){$2$}
\put(90,58){$2$}
\put(133,43){$3$}
\put(149,71){$5$}
\put(113,68){$4$}
\put(119,76){$6$}

\put(115,43){\color{blue}$1$}
\put(138,85){\color{blue}$2$}
\put(90,85){\color{blue}$3$}
\put(70,45){\color{blue}$4$}

\put(279,101){$8$}
\put(254,58){$8$}
\put(297,43){$3$}
\put(313,71){$5$}
\put(278,68){$4$}
\put(283,76){$6$}

\put(279,43){\color{blue}$5$}
\put(302,85){\color{blue}$6$}
\put(254,85){\color{blue}$7$}
\put(234,45){\color{blue}$8$}

\end{picture}
\caption{The polyhedra $M_{C/\iota}^{++} \cup M_{C/\iota}^{-+}$ (left) and $M_{C/\iota}^{+-} \cup M_{C/\iota}^{--}$ (right).}
\label{figure:type4_two_octahedra_glued}
\end{figure}
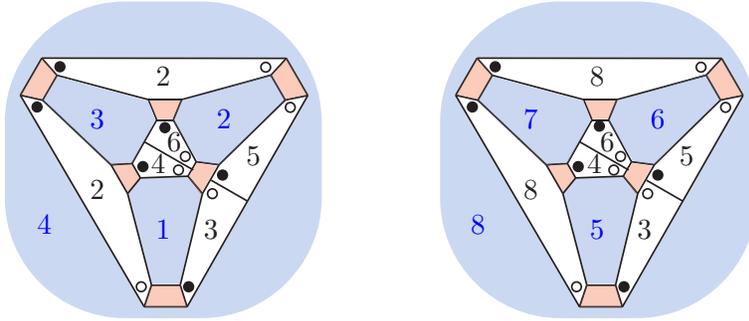
After renumbering the faces, we get a decomposition of $\pi^{-1}(C)$ as shown in 
Figure \ref{figure:type4_decomposition}.

The resulting 3-manifold $\pi^{-1}(C)$ by gluing the faces of these two truncated octahedra 
is a genus-$3$ handlebody. 
From the way of gluing, it is easily checked that 
on the boundary of this genus-$3$ handlebody, there exist 
two three-holed spheres, each of which consists of two faces of these truncated octahedra, 
and a single four-holed spheres, 
which consists of four faces as shown on the right in Figure \ref{figure:type4_decomposition}.  
\end{proof}

\begin{lemma}
\label{lem:type-5}
Let $c$ be a vertex of $X_P$ of Type $5$-$3$. 
Then $\pi^{-1}(C)$ is a genus-$2$ handlebody obtained by gluing the faces of a single truncated ideal regular hyperbolic octahedron 
as shown in Figure $\ref{figure:type5_decomposition}$. 
This genus-$2$ handlebody has two three-holed spheres on the boundary, 
each of which consists of two faces of this truncated octahedron.
\end{lemma}
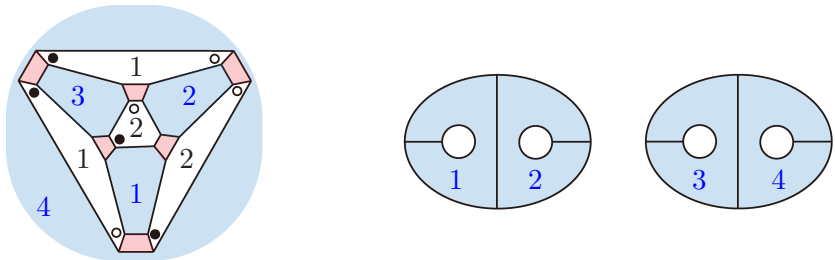
\begin{figure}[htbp]
\centering\includegraphics[width=11cm]{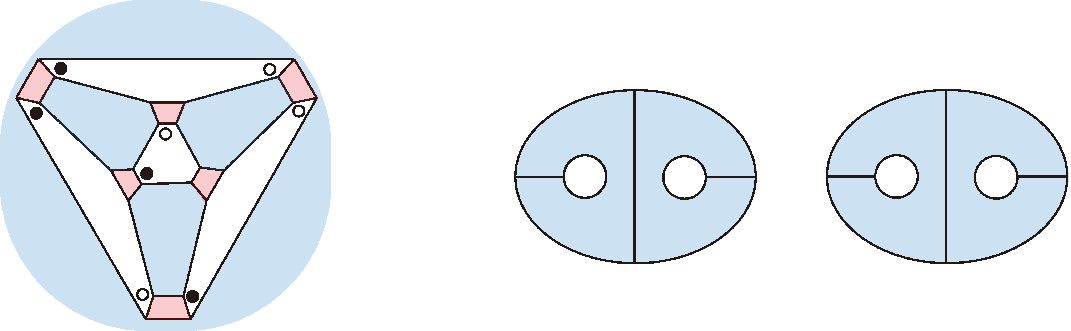}
\begin{picture}(400,0)(0,0)

\put(90,83){$1$}
\put(70,48){$1$}
\put(109,48){$2$}
\put(90,60){$2$}

\put(90,35){\color{blue}$1$}
\put(110,72){\color{blue}$2$}
\put(68,72){\color{blue}$3$}
\put(55,30){\color{blue}$4$}

\put(211,40){\color{blue}$1$}
\put(241,40){\color{blue}$2$}

\put(303,40){\color{blue}$3$}
\put(333,40){\color{blue}$4$}
\end{picture}
\caption{A polyhedral decomposition of $\pi^{-1}(C)$ for a Type $5$-$3$ vertex $c$.}
\label{figure:type5_decomposition}
\end{figure}

\begin{proof}

In this case, $M_{C/\iota}$ is the quater of a truncated ideal regular hyperbolic octahedron shown in Figure \ref{figure:identify_3thickening_polyhedron}, 
and the decomposition $\pi^{-1}(C) = M_{C / \iota}^{++} \cup M_{C / \iota}^{+-} \cup M_{C / \iota}^{-+} \cup M_{C / \iota}^{--}$ 
is as in Figure \ref{figure:type5_four_octahedra}. 
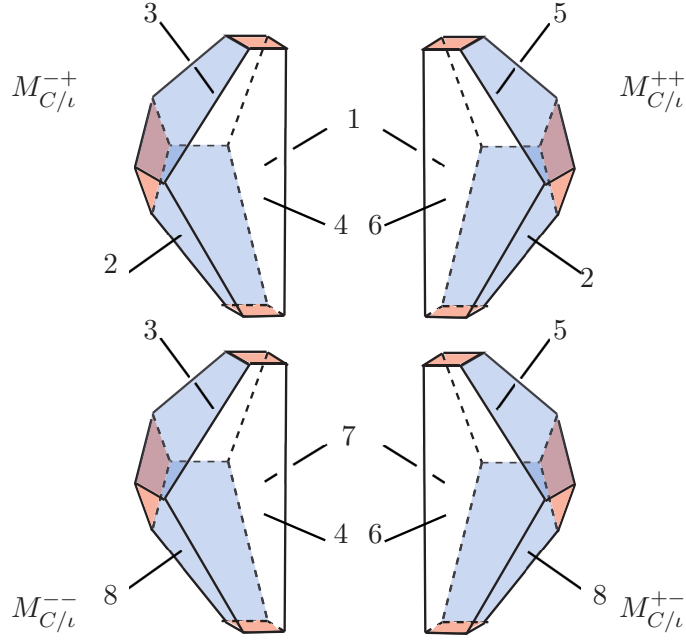
\begin{figure}[htbp]
\centering\includegraphics[width=6cm]{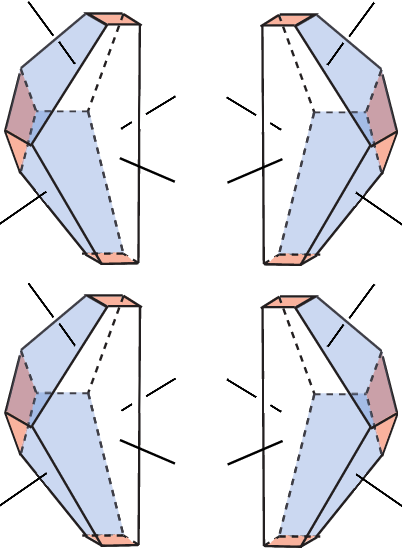}
\begin{picture}(400,0)(0,0)
\put(120,245){$3$}
\put(105,150){$2$}
\put(192,165){$4$}
\put(197,205){$1$}

\put(120,125){$3$}
\put(105,25){$8$}
\put(192,48){$4$}
\put(195,85){$7$}

\put(275,245){$5$}
\put(285,145){$2$}
\put(205,165){$6$}

\put(275,125){$5$}
\put(290,25){$8$}
\put(205,48){$6$}

\put(300,217){$M_{C/\iota}^{++}$}
\put(300,20){$M_{C/\iota}^{+-}$}
\put(70,217){$M_{C/\iota}^{-+}$}
\put(70,20){$M_{C/\iota}^{--}$}

\end{picture}
\caption{The decomposition $\pi^{-1}(C)=M_{C/\iota}^{++} \cup M_{C/\iota}^{+-} \cup M_{C/\iota}^{-+} \cup M_{C/\iota}^{--}$.}
\label{figure:type5_four_octahedra}
\end{figure}
Glue together the two polyhedra $M_{C/\iota}^{++}$ and $M_{C/\iota}^{-+}$ along the faces assigned $1$, 
and glue together $M_{C/\iota}^{+-}$ and $M_{C/\iota}^{--}$ along the faces assigned $7$, according to Figure \ref{figure:type5_four_octahedra}. 
Then, we get the two polyhedra $M_{C/\iota}^{++} \cup M_{C/\iota}^{-+}$ and $M_{C/\iota}^{+-} \cup M_{C/\iota}^{--}$ shown in Figure \ref{figure:type5_two_octahedra_glued}. 
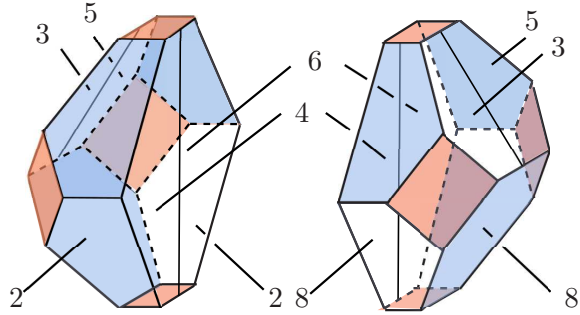
\begin{figure}[htbp]
\centering\includegraphics[width=7cm]{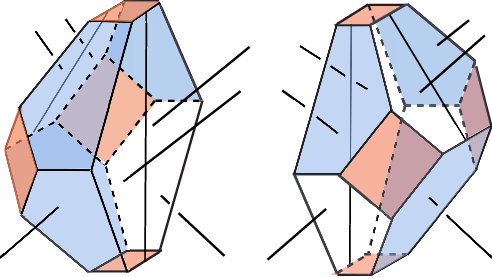}
\begin{picture}(400,0)(0,0)
\put(105,115){$3$}
\put(123,123){$5$}
\put(95,15){$2$}

\put(208,105){$6$}
\put(203,85){$4$}
\put(193,15){$2$}

\put(290,120){$5$}
\put(300,110){$3$}
\put(203,15){$8$}
\put(305,15){$8$}

\end{picture}
\caption{The polyhedra $M_{C/\iota}^{++} \cup M_{C/\iota}^{-+}$ (left) and $M_{C/\iota}^{+-} \cup M_{C/\iota}^{--}$ (right).}
\label{figure:type5_two_octahedra_glued}
\end{figure}
Furthermore, glue together these two big polyhedra $M_{C/\iota}^{++} \cup M_{C/\iota}^{-+}$ and $M_{C/\iota}^{+-} \cup M_{C/\iota}^{--}$ 
along the faces assigned $4$ and $6$, which are adjacent each other, to get a single polyhedron shown in Figure \ref{figure:type5_four_octahedra_glued}, 
which is a truncated ideal regular hyperbolic octahedron.
\begin{figure}[htbp]
\centering\includegraphics[width=5cm]{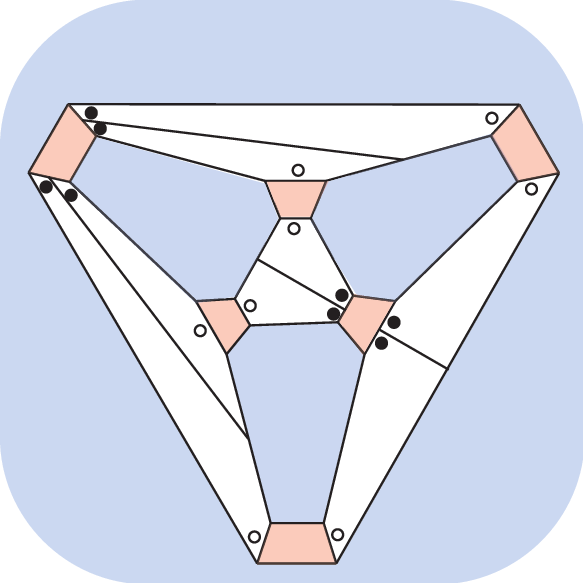}
\begin{picture}(400,0)(0,0)
\put(212,120){$2$}
\put(190,113){$8$}
\put(163,83){$8$}
\put(169,59){$2$}
\put(218,50){$3$}
\put(239,85){$5$}
\put(201,88){$5$}
\put(195,78){$3$}

\put(198,50){\color{blue} $1$}
\put(228,102){\color{blue} $2$}
\put(168,102){\color{blue} $3$}
\put(145,45){\color{blue} $4$}

\end{picture}
\caption{The polyhedron $(M_{C/\iota}^{++} \cup M_{C/\iota}^{-+}) \cup (M_{C/\iota}^{+-} \cup M_{C/\iota}^{--})$.}
\label{figure:type5_four_octahedra_glued}
\end{figure}
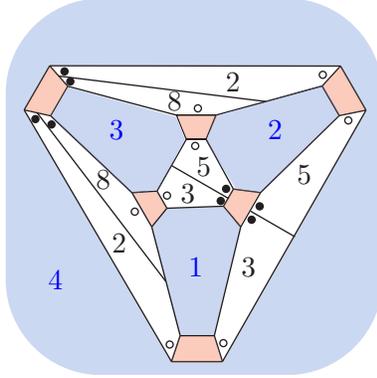

The resulting 3-manifold $\pi^{-1}(C)$ by the gluing in this case is a genus-$2$ handlebody. 
From the way of gluing, this genus-$2$ handlebody has 
two three-holed spheres on the boundary, 
each of which consists of two faces of this truncated octahedron as shown on the right in Figure \ref{figure:type5_decomposition}.
\end{proof}

For a vertex $c \in X_P$ of Type $3$, which is the final case, the situation is slightly more complicated. 
In fact, a direct construction of the decomposition using the preimage of $\pi$ as in 
Lemmas \ref{lem:type-1}--\ref{lem:type-5} seems to be hard. 
Instead, we give a decomposition of $\pi^{-1}(C)$ using 
a specific divide link as follows. 

\begin{lemma}
\label{lem:type-3}
Let $c$ be a vertex of $X_P$ of Type $3$. 
Then $\pi^{-1}(C)$ is a genus-$3$ handlebody obtained by gluing the faces of ten truncated ideal regular hyperbolic tetrahedra 
as shown in Figure $\ref{figure:ten_tetrahedra_for_A2}$. 
This genus-$3$ handlebody has four three-holed spheres on the boundary, 
each of which consists of two faces of these truncated tetrahedra.
\end{lemma}
\begin{proof}
Let $P_0$ be the divide in Example \ref{ex:divide}. 
Recall that the associated links $L_{P_0}$, $L_{X_{P_0}}$ and $L_{\widehat{X_{P_0}}}$ are as depicted in Figure \ref{figure:divide_l1_l2}. 

\begin{claim*}
\label{claim: exteriors of hatLP and L}
The link $L_{X_{P_0}}$ is hyperbolic, and the exterior of $L_{X_{P_0}}$, equivalently, 
the $3$-manifold $M_{P_0}$, 
can be decomposed into ten truncated ideal regular hyperbolic tetrahedra as shown in Figure $\ref{figure:ten_tetrahedra_for_L}$. 
Furthermore, each of the two three-holed spheres $Q_1$ and $Q_2$ shown on the left in Figure $\ref{figure:Y1_and_Y2}$ consists of two faces of 
that truncated tetrahedra. 
\end{claim*}
\begin{figure}[htbp]
\centering\includegraphics[width=14cm]{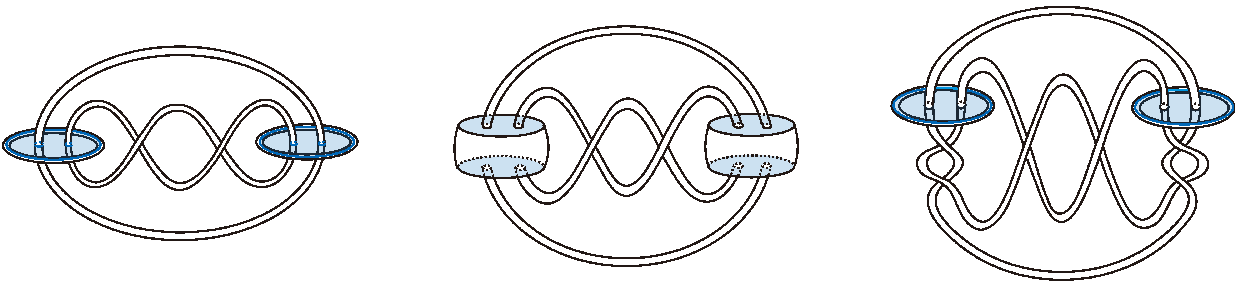}
\begin{picture}(400,0)(0,0)
\end{picture}
\caption{Cutting and twisting around disks.}
\label{figure:cutting_and_twisting}
\end{figure}
\begin{proof}[Proof of Claim]
Let $L$ be the link shown on the left in Figure \ref{figure:shadow_of_L}, the hyperbolic structure of whose exterior is explained in Example \ref{ex:link_10v4}.
We are going to show that the exterior of $L_{X_{P_0}}$ is homeomorphic to $E(L)$. 
Cut the exterior $E(L)$ along three-holed spheres as shown in Figure \ref{figure:cutting_and_twisting}, where the three-holed spheres are painted in blue, 
and then, reglue the them after twisting one side by $2 \pi$. 
The resulting space is clearly homeomorphic to $E(L)$, but this can be thought of as the exterior of another link $L'$ shown on the left in Figure \ref{figure:hatLP_and_its_twisting}.
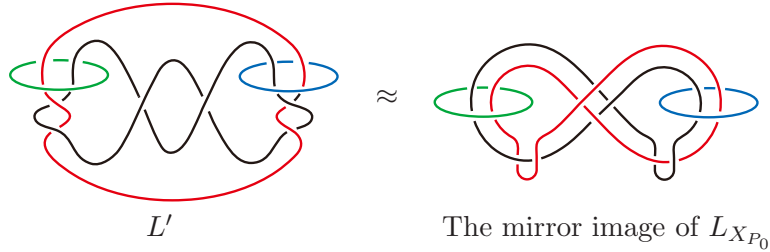
\begin{figure}[htbp]
\centering\includegraphics[width=10cm]{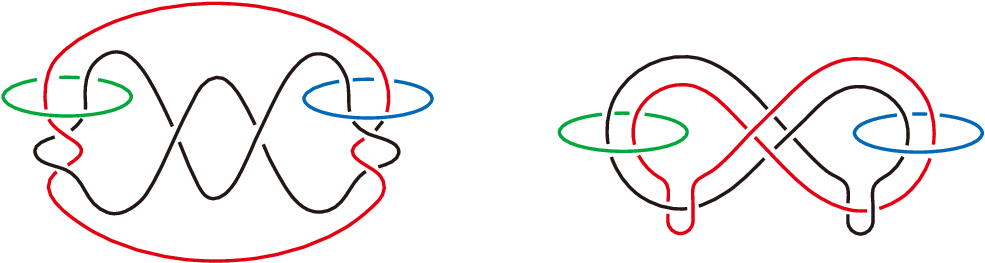}
\begin{picture}(400,0)(0,0)
\put(110,0){$L'$}
\put(197,50){$\approx$}
\put(222,0){The mirror image of $L_{X_{P_0}}$}
\end{picture}
\caption{The link $L'$ and the mirror image of $L_{X_{P_0}}$ are isotopic.}
\label{figure:hatLP_and_its_twisting}
\end{figure}
The link $L'$ can be moved by an isotopy to the one shown on the right in Figure \ref{figure:hatLP_and_its_twisting}. 
This is the mirror image of the diagram of $L_{X_{P_0}}$.  

As seen in Example \ref{ex:link_10v4} , $L$ is a hyperbolic link, and the exterior $E(L)$ of $L$ has a decomposition into ten truncated ideal regular hyperbolic tetrahedra 
as shown in Figure \ref{figure:ten_tetrahedra_for_L}. 
Moreover, each of the three-holed spheres in $E(L)$ taken in the above argument consists of $2$ faces of the truncated tetrahedra. 
Therefore, by the above procedure that gives the correspondence between $E(L)$ and $E(L_{X_{P_0}})$,  the latter assertion in Claim follows 
immediately. 
\end{proof}

Let $Y_1$, $Y_2$ be tripods in $\Nbd (P_0; X_{P_0} )$ shown in Figure \ref{figure:Y1_and_Y2}. 
Then both of the preimages $Q_i := \partial \pi^{-1} (Y_i)$ ($i=1,2$) are three-holed spheres. 
A regular neighborhood $\Nbd (c; X_{P_0})$ of $c$ is then obtained by cutting $\Nbd (P_0; X_{P_0})$ along $Y_1$ and $Y_2$. 
Thus, the $3$-manifold $\pi^{-1}(C)$ is obtained by cutting $M_{P_0}$ along $Q_1$ and $Q_2$. 
By the above Claim, 
$E (L_{X_{P_0}})$ has a decomposition into ten truncated ideal regular hyperbolic tetrahedra 
such that $Q_1$, $Q_2$ consist of their faces. 
Hence, this decomposition induces the decomposition of $\pi^{-1}(C)$ into the same truncated ideal regular hyperbolic tetrahedra, 
whose gluing information are shown in Figure \ref{figure:ten_tetrahedra_for_A2}. 

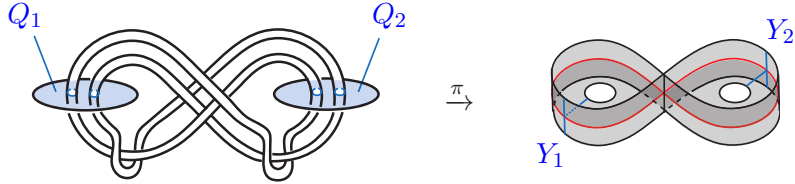
\begin{figure}[htbp]
\centering\includegraphics[width=10cm]{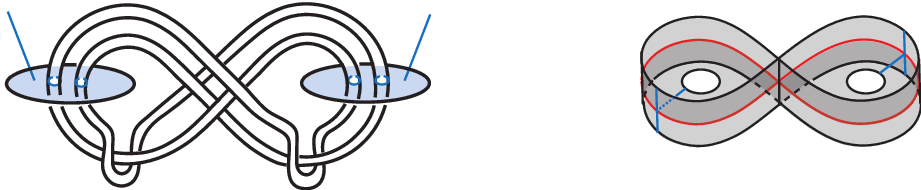}
\begin{picture}(400,0)(0,0)
\put(50,73){\color{blue} $Q_1$}
\put(188,73){\color{blue} $Q_2$}

\put(215,40){$\xrightarrow{\pi}$}

\put(250,20){\color{blue} $Y_1$}
\put(337,67){\color{blue} $Y_2$}
\end{picture}
\caption{The tripods $Y_1$, $Y_2$ and their preimages $Q_1$, $Q_2$.}
\label{figure:Y1_and_Y2}
\end{figure}

The resulting $3$-manifold by the gluing is a genus-$3$ handlebody having four three-holed spheres on the boundary, each of which consists of two faces of the truncated tetrahedra as shown in Figure \ref{figure:holed_spheres_for_A1}.
\end{proof}

\begin{proof}[Proof of Theorem $\ref{theo:main}$]
Let $c_1, c_2,\ldots ,c_n$ be the double points of $P$. 
From Lemmas \ref{lem:type-1}--\ref{lem:type-3}, 
for each vertex $c_i \in X_P$, the $3$-manifold $\pi^{-1}(C_i)$ is a handlebody equipped with 
a decomposition into truncated ideal regular hyperbolic polyhedra. 
Note that the sum of the dihedral angles of the hyperbolic polyhedra around each edge in the interior of each handlebody is $2 \pi$. 
For an edge on the boundary of $\pi^{-1}(C_i)$, the following holds. 
When the two sides of the edge on the boundary of $\pi^{-1}(C_i)$ are both cusp faces or both gluing faces, 
then exactly two face corners of the polyhedra meet there. 
When the one side of the edge on the boundary of $\pi^{-1}(C_i)$ is a cusp face and the other is a gluing face, 
then exactly one face corner of the polyhedra is there.  
 
Glue the handlebodies $\pi^{-1}(C_i)$ ($i=1, 2, \ldots, n $) together along the holed spheres on their boundaries 
according to the combinatorial structure of $X_P$ to get $M_P$. 
From the above note, the sums of the dihedral angles gathered on edges in the interior of $M_P$ are all $2 \pi$.
Furthermore, by pre-adjusting the scale of the facets (i.e., the height of the horosphere of the ideal vertex) of the 
truncated ideal regular hyperbolic polyhedra according to the Types of the corresponding vertices $c_1, c_2,\ldots ,c_n$ appropriately,  
we can assume that each two spheres with holes on the boundary of $\pi^{-1}(C_{1}), \pi^{-1}(C_{2}), \ldots , \pi^{-1}(C_{n})$ that are to be
glued together are actually isometric. 
Therefore, we get a tessellation of each boundary torus of $M_{P}$ by cusp faces of  $\pi^{-1}(C_{i})$ ($i=1, 2, \ldots, n $), 
which implies that each component of $\partial M_{P}$ is naturally equipped with an Euclidean structure.
Hence, the above decomposition directly gives a decomposition of $M_P$ into truncated ideal regular hyperbolic polyhedra. 
This induces a decomposition of $\Int ( M_P ) $ into ideal hyperbolic regular polyhedra. 
The volume formula in Theorem \ref{theo:main} now follows from Lemmas \ref{lem:type-1}--\ref{lem:type-3}. 
\end{proof}

\begin{corollary}
\label{cor:main}
Let $P \subset D$ be a connected prime divide with at least one double point. 
Let $n_1$, $n_2$, $n_3$, $n_4$ and $n_5$ be the number of its double points of 
Types $1$, $2$, $3$, $4$-$2$ and $5$-$3$, respectively.  
Then the hyperbolic volume of $L_P \subset S^3$ is less than 
\[ 10 n_3 v_{\mathrm{tet}} + (4n_1 + 2n_4 + n_5) v_{\mathrm{oct}} + n_2 v_{\mathrm{cuboct}}, \]
where the hyperbolic volume of a non-hyperbolic links is defined to be zero. 
\end{corollary}


\begin{example}
\label{ex:divide_link_v8}
Let $P_1$ be  the divide shown on the left top in Figure \ref{figure:divide_link_v8}. 
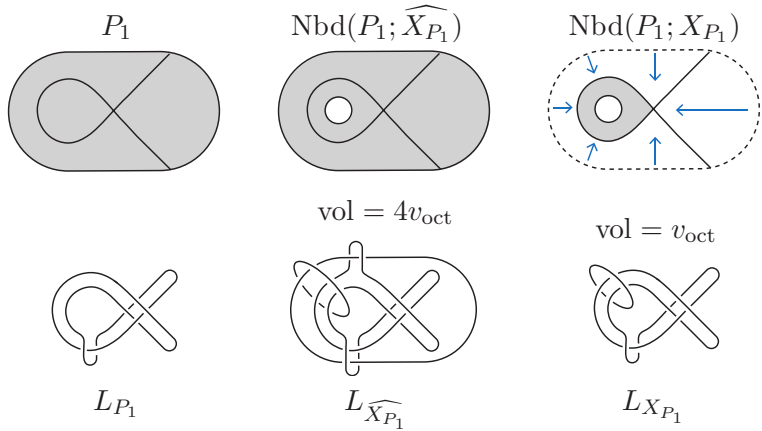
\begin{figure}[htbp]
\centering\includegraphics[width=10cm]{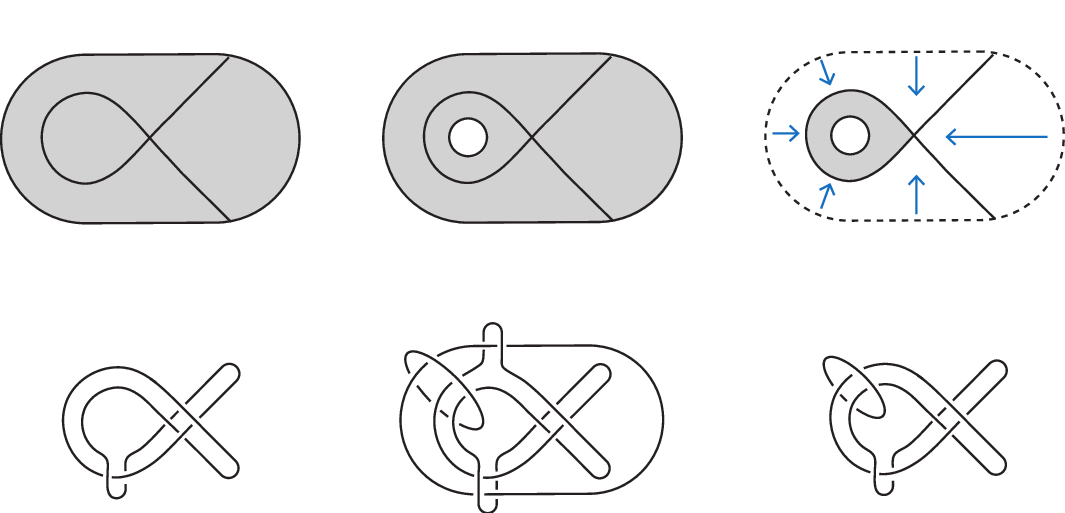}
\begin{picture}(400,0)(0,0)

\put(90,0){$L_{P_1}$}
\put(185,0){$L_{\widehat{X_{P_1}}}$}
\put(290,0){$L_{X_{P_1}}$}

\put(93,142){$P_1$}
\put(165,142){$\Nbd ( P_1 ; \widehat{X_{P_1}} )$}
\put(270,142){$\Nbd ( P_1 ; X_{P_1} )$}

\put(175,72){$\vol = 4 v_{\mathrm{oct}}$}
\put(280,65){$\vol = v_{\mathrm{oct}}$}
\end{picture}
\caption{The divide $P_1$ and the associated links $L_{P_1}$, $L_{\widehat{X_{P_1}}}$ and $L_{X_{P_1}}$.}
\label{figure:divide_link_v8}
\end{figure}
The link $L_{P_1}$ of $P_1$ is shown on the left bottom of the figure. 
The middle top of Figure \ref{figure:divide_link_v8} illustrates a simplified figure of the polyhedron $\Nbd ( P_1 ; \widehat{X_{P_1}} )$. 
Below that, a diagram of the link $L_{\widehat{X_{P_1}}}$ is drawn. 
Recall that the preimage $\pi^{-1} (\widehat{X_{P_1}})$ is homeomorphic to $E (L_{\widehat{X_{P_1}}})$. 
Since $\widehat{X_{P_1}}$ has a single vertex, which is of Type 1, $L_{\widehat{X_{P_1}}}$ is a hyperbolic link of volume $4 v_{\mathrm{oct}}$. 
The right top of Figure \ref{figure:divide_link_v8} is then a simplified figure of the polyhedron $\Nbd ( P_1 ; X_{P_1} )$, whose single vertex of Type $5$-$3$. 
Thus, by Theorem \ref{theo:main}, $\Int M_{P_1}$ has a hyperbolic manifold of volume $v_{\mathrm{oct}}$. 
In other words, the link $L_{X_{P_1}}$ shown on the right bottom in Figure \ref{figure:divide_link_v8} is a hyperbolic link of volume $v_{\mathrm{oct}}$.

\end{example}

\begin{example}
Let $P_2$ be  the divide shown on the left top in Figure \ref{figure:divide_link_2v48}. 
\begin{figure}[htbp]
\centering\includegraphics[width=10cm]{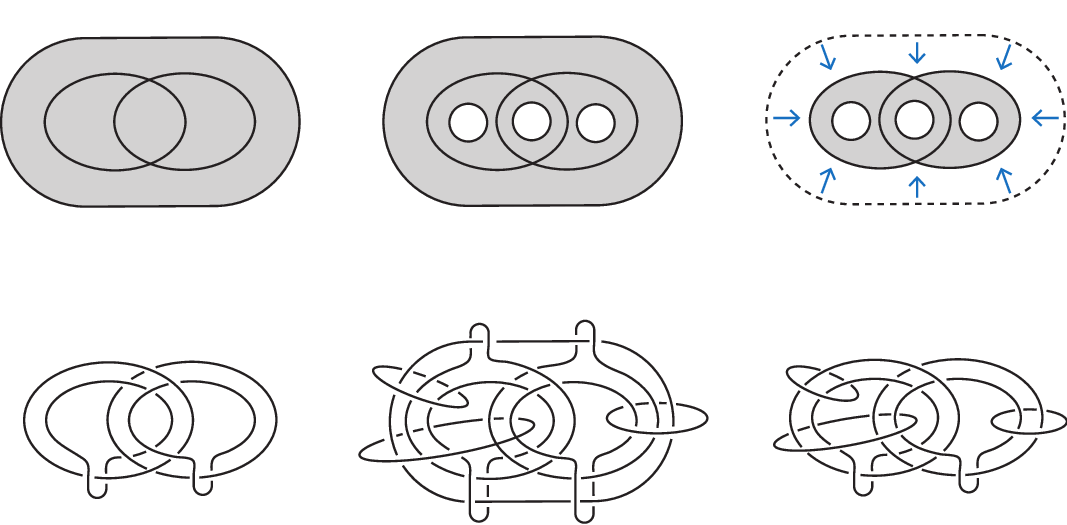}
\begin{picture}(400,0)(0,0)

\put(90,0){$L_{P_2}$}
\put(185,0){$L_{\widehat{X_{P_2}}}$}
\put(290,0){$L_{X_{P_2}}$}

\put(93,149){$P_2$}
\put(168,149){$\Nbd ( P_2 ; \widehat{X_{P_2}} )$}
\put(271,149){$\Nbd ( P_2 ; X_{P_2} )$}

\put(175,74){$\vol = 8 v_{\mathrm{oct}}$}
\put(268,67){$\vol = 2 v_{\mathrm{cuboct}}$}
\end{picture}
\caption{The divide $P_2$ and the associated links $L_{P_2}$, $L_{\widehat{X_{P_2}}}$ and $L_{X_{P_2}}$.}
\label{figure:divide_link_2v48}
\end{figure}
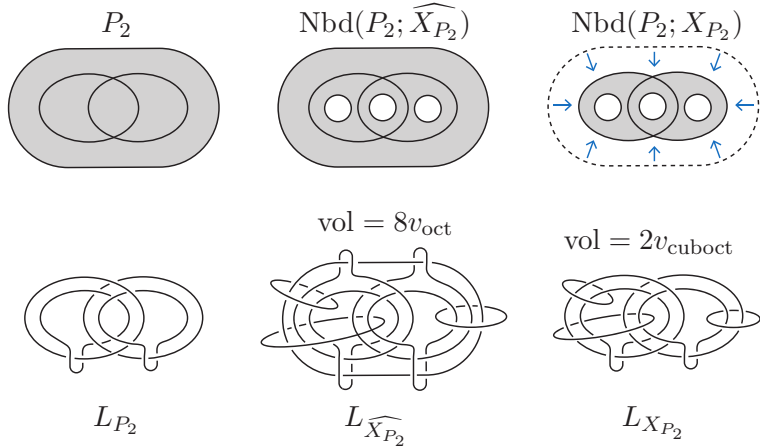
The link $L_{P_2}$ of $P_2$ is shown on the left bottom of the figure. 
The middle top of Figure \ref{figure:divide_link_2v48} illustrates a simplified figure of the polyhedron $\Nbd ( P_2 ; \widehat{X_{P_2}} )$. 
Below that, a diagram of the link $L_{\widehat{X_{P_2}}}$ is drawn. 
Since $\widehat{X_{P_2}}$ has two vertices, both of which are of Type 1, $L_{\widehat{X_{P_2}}}$ is a hyperbolic link of volume $8 v_{\mathrm{oct}}$. 
The right top of Figure \ref{figure:divide_link_2v48} is then a simplified figure of the polyhedron $\Nbd ( P_2 ; X_{P_2} )$, which has two Type $2$ vertices. 
Thus, by Theorem \ref{theo:main}, the link $L_{X_{P_2}}$ shown on the right bottom in Figure \ref{figure:divide_link_2v48} is a hyperbolic link of volume $2 v_{\mathrm{cuboct}}$.
\end{example}

\begin{example}
\label{ex:divide_link_generic}
Let $P_3$ be  the divide shown on the left in Figure \ref{figure:divide_link_generic}. 
\begin{figure}[htbp]
\centering\includegraphics[width=11cm]{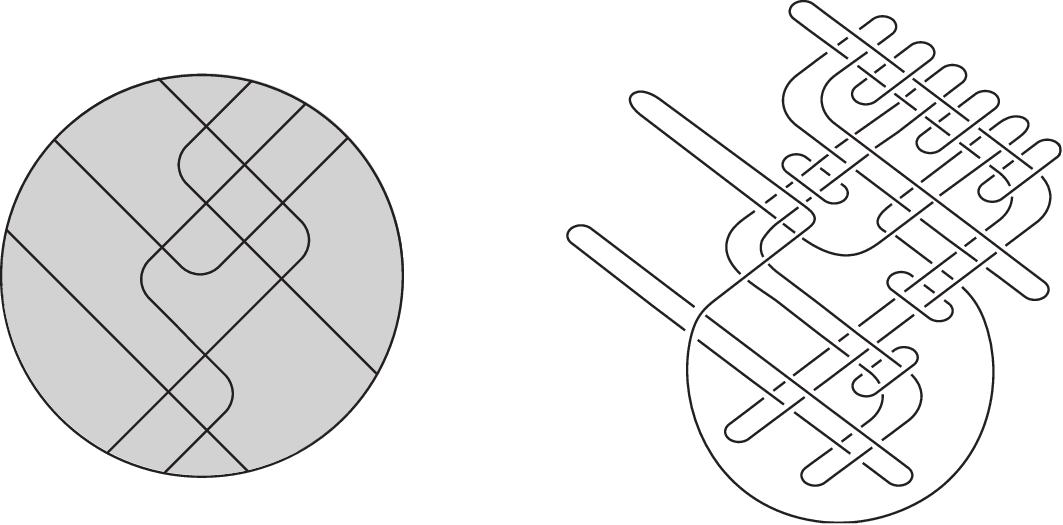}
\begin{picture}(400,0)(0,0)

\put(93,0){$P_3$}
\put(280,0){$L_{X_{P_3}}$}

\end{picture}
\caption{The divide $P_3$ and the associated link $L_{X_{P_3}}$.}
\label{figure:divide_link_generic}
\end{figure}
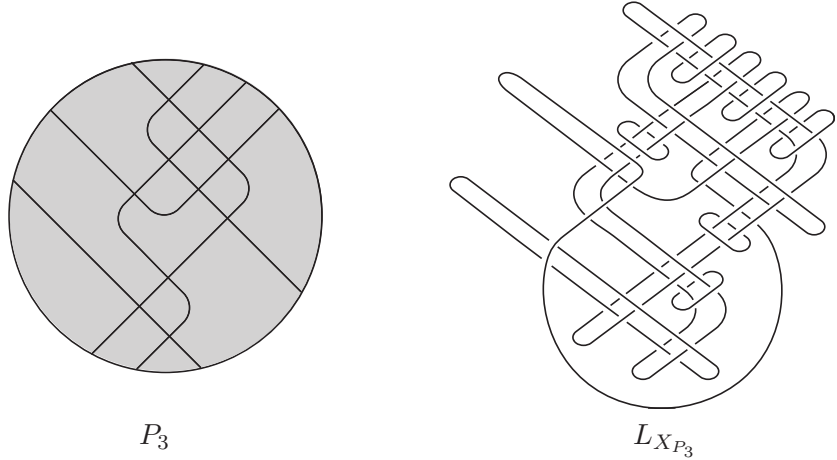
The polyhedron $\Nbd ( P_3 ; X_{P_3} )$ has one vertex each of Types 1, 2 and 3, four vertices of Type 4-2 and three vertices of Type 5-3. 
Then, by Theorem \ref{theo:main}, $\Int M_{P_3}$ has a hyperbolic manifold of volume 
\[ 1 \cdot 4 v_{\mathrm{oct}} + 1 \cdot v_{\mathrm{cuboct}} + 1 \cdot 10 v_{\mathrm{tet}} + 4 \cdot 2 v_{\mathrm{oct}} + 3 \cdot v_{\mathrm{oct}} = 
10 v_{\mathrm{tet}} + 15 v_{\mathrm{oct}} + v_{\mathrm{cuboct}} = 77.1534 \cdots \]


\end{example}

The following theorem shows that the upper bound of the volumes of links of divides given in Corollary \ref{cor:main} is asymptotically sharp. 

\begin{theorem}
\label{thm: the upper boune is sharp}
There exist a sequence $\{ P_n \}_{n \in \NN}$ of divides satisfying the following: 
 \begin{enumerate}
 \item
The polyhedron $X_{P_n}$ has $n$ vertices, and their types are all $5$-$3$. 
 \item
 The volume of $L_{P_n}$ is equal to $n v_{\mathrm{oct}}$ asymptotically, that is, we have 
 \[\lim_{n \to \infty} \frac{\vol (S^3 - L_{P_n})}{n v_{\mathrm{oct}}} = 1 , \]  
where $\vol (S^3 - L_{P_n})$ is the hyperbolic volume of $S^3 - L_{P_n}$. 
 \end{enumerate}
\end{theorem}

\begin{proof}
Let $P_n$ be the divide shown on the left in Figure \ref{figure:chain_link}. 
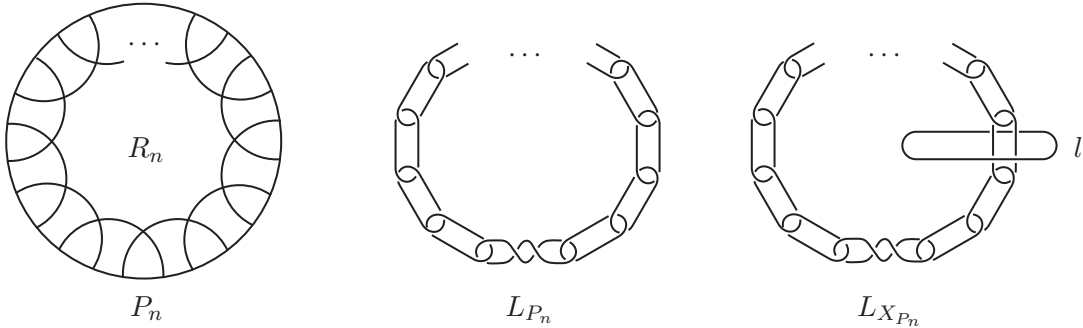
\begin{figure}[htbp]
\centering\includegraphics[width=14cm]{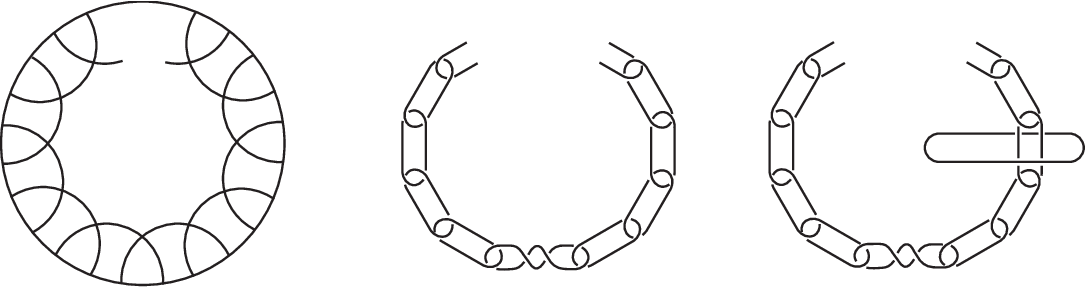}
\begin{picture}(400,0)(0,0)
\put(47,100){$\cdots$}
\put(191,96){$\cdots$}
\put(327,96){$\cdots$}
\put(48,0){$P_n$}
\put(190,0){$L_{P_n}$}
\put(323,0){$L_{X_{P_n}}$}
\put(405,60){$l$}
\put(47,60){$R_n$}
\end{picture}
\caption{A divide $P_n$ and the associated links $L_{P_n}$ and $L_{X_{P_n}}$.}
\label{figure:chain_link}
\end{figure}
Note that $P_n$ is the $n$-fold covering of the divide in Example \ref{ex:divide_link_v8} branched over a point in the $1$-gon region. 
Then, clearly,  $X_{P_n}$ has $n$ vertices $c_1, c_2 \ldots , c_n$, and their types are all $5$-$3$.  
The links $L_{P_n}$ and $L_{X_{P_n}}= L_{P_n} \cup l$ are as drawn on the middle and right, respectively, in Figure \ref{figure:chain_link}. 
By Theorem $\ref{theo:main}$, the link $L_{X_{P_n}}$ is hyperbolic, and the volume of $\Int M_{P_n} = S^3 - L_{X_{P_n}}$ is $n v_{\mathrm{oct}}$. 
The link $L_{P_n}$ is an $n$-chain link. 
Thus,  by Neumann-Reid \cite{NR92}, $L_{P_n}$ is also hyperbolic for $n \geq 5$.  
Let $R_n$ be the unique internal region of $P_n$. 
Define a simple closed curve $r$ by $r = R_n \cap \partial \Nbd (P_n ; X_{P_n})$, and denote by 
$T_r$ the torus component of $\partial M_{P_n}$ corresponding to $r$, namely
$T_r = \pi^{-1} (r)$. 
For each vertex $c_i$, the 3-manifold $\pi^{-1}(C_i)$ is described in Lemma \ref{lem:type-5}. 
In particular, $\pi^{-1}(C_i)$ is a genus-$2$ handlebody obtained by gluing faces of a truncated ideal regular hyperbolic octahedron each other as shown in Figure \ref{figure:type5_decomposition}. 
For our purpose, it is more prospective to draw each piece $\pi^{-1}(C_i)$ as on the right side in Figure \ref{figure:chain_link_vertex}, 
see Yoshida \cite{Yos21}.
\begin{figure}[htbp]
\centering\includegraphics[width=11cm]{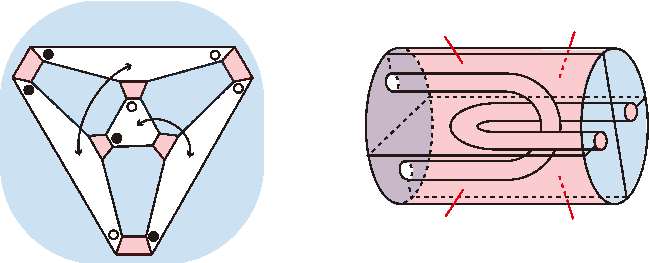}
\begin{picture}(400,0)(0,0)
\put(105,45){\color{blue}$1$}
\put(130,90){\color{blue}$2$}
\put(78,90){\color{blue}$3$}
\put(60,40){\color{blue}$4$}

\put(105,94){\color{red}$1$}
\put(90,67){\color{red}$2$}
\put(120,67){\color{red}$3$}
\put(105,19){\color{red}$4$}
\put(155,104){\color{red}$5$}
\put(55,104){\color{red}$6$}

\put(232,85){\color{blue}$1$}
\put(235,68){\color{blue}$2$}
\put(332,57){\color{blue}$4$}
\put(345,70){\color{blue}$3$}

\put(250,123){\color{red}$4$}
\put(253,25){\color{red}$5$}
\put(318,25){\color{red}$1$}
\put(320,127){\color{red}$2$}

\put(190,77){$=$}
\end{picture}
\caption{The piece $\pi^{-1}(C_i)$.}
\label{figure:chain_link_vertex}
\end{figure}
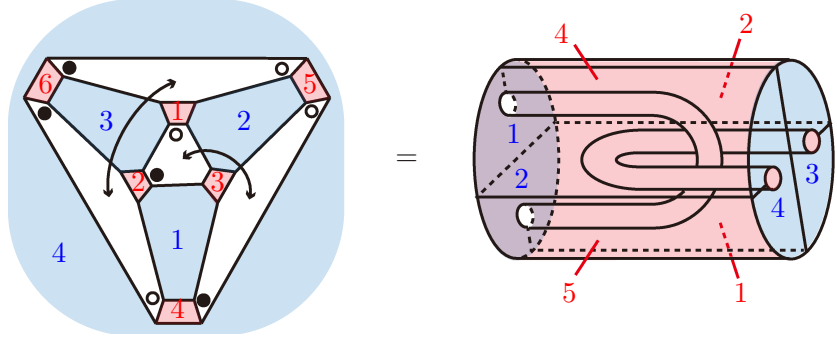
Recall that the ideal regular hyperbolic octahedron has a maximal horocusp section consisting of six Euclidean unit squares. 
Thus, we can assume that each cusp face of the truncated ideal regular hyperbolic octahedron is the unit square. 
Notice that the annulus on the outer side of the right figure in Figure \ref{figure:chain_link_vertex}, which corresponds to 
the component $l$, consists of four of these squares: those labeled $1, 2, 4, 5$. 
The 3-manifold $M_{P_n}$ is obtained by gluing $\pi^{-1}(C_1), \pi^{-1}(C_2), \ldots, \pi^{-1}(C_n)$ according to the combinatorial structure of $X_{P_n}$. 
Then, the boundary torus $T_r$ is equipped with the Euclidean structure that consists of $4n$ Euclidean unit squares. 
By comparing the diagram of $L_{X_{P_n}}$ on the right in Figure \ref{figure:chain_link} and the union of the pieces $\pi^{-1}(C_i)$'s drawn on 
the right in Figure \ref{figure:chain_link_vertex}, we can find the meridian $\mu$ of $l$ on the torus, and in fact, we see that the cusp shape 
of  $L_{X_{P_n}}$ corresponding to the component $l$ is as in Figure \ref{figure:chain_link_cusp_shape}. 
\begin{figure}[htbp]
\centering\includegraphics[width=6cm]{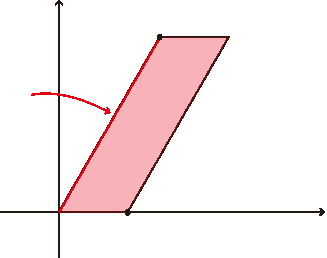}
\begin{picture}(400,0)(0,0)
\put(135,25){$0$}
\put(178,25){$4$}
\put(166,135){$n-4 + \sqrt{-1} n$}
\put(45,95){\color{red} the meridian of $l$}
\put(260,120){$\CC$}
\end{picture}
\caption{The cusp shape of $L_{X_{P_n}}$ corresponding to the component $l$.}
\label{figure:chain_link_cusp_shape}
\end{figure}
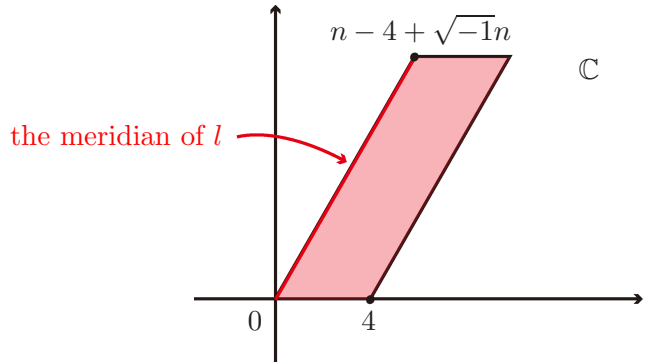
In particular, the length $\slope (R_n)$ of the meridian $\mu$ is $\sqrt{(n-4)^2 + n^2}$ with respect to the Euclidean structure. 
The exterior $E (L_{P_n})$ of the link $L_{P_n}$ is obtained by performing Dehn filling along $T_l$ whose slope is exactly $\mu$.  
Thus, the length $\slope (R_n)$ of the filling slope is $\sqrt{(n-4)^2 + n^2}$. 
By Futer-Kalfagianni-Purcell \cite{FKP08}, we have the inequalities
\[
\left( 1 - \left( \frac{2 \pi}{ \slope (R_n) } \right)^2 \right)^{3/2}  n v_{\mathrm{oct}} 
\leq 
\vol (S^3 - L_{P_n}) 
< n v_{\mathrm{oct}} , 
\] 
which implies that $ \vol (S^3 - L_{P_n}) / n v_{\mathrm{oct}} $ approaches to $1$ as $n$ goes to infinity. 
\end{proof}

\section*{Acknowledgments} 
The authors wish to express their gratitude to 
Masaharu Ishikawa, Kazuhiro Ichihara and Hironobu Naoe for their very helpful suggestions and comments.

\end{document}